\documentclass[12pt, ams fonts]{amsart}

\usepackage{color}

\usepackage{stmaryrd}
\usepackage{tikz-cd}
\usepackage{mathtools}
\usepackage{amssymb,amsmath, amsthm}
\usepackage{amsfonts}
\usepackage{graphicx}
\usepackage{tikz-cd}
\usepackage{graphicx}
\usepackage{yhmath}
\usepackage{mathdots}
\usepackage{MnSymbol}

\usepackage[vcentermath,enableskew]{youngtab}

\input xy
\xyoption{all}

\font\teneufm=eufm10 \font\seveneufm=eufm7 \font\fiveeufm=eufm5 
\newfam\eufmfam
\textfont\eufmfam=\teneufm \scriptfont\eufmfam=\seveneufm
\scriptscriptfont\eufmfam=\fiveeufm

\newcommand{\C}{\mathbb{C}}

\newcommand{\Z}{\mathbb{Z}}

\DeclareMathOperator{\Supp}{Supp}
\DeclareMathOperator{\Span}{Span} 
\DeclareMathOperator{\Hom}{Hom}

 \DeclareMathOperator{\Aut}{Aut}

\DeclareMathOperator{\Tr}{Tr}

\DeclareMathOperator{\re}{Re}

\DeclareMathOperator{\diag}{diag}
\DeclareMathOperator{\Mat}{Mat}
\DeclareMathOperator{\GL}{GL}
\DeclareMathOperator{\Res}{Res}
\DeclareMathOperator{\I}{I}
\DeclareMathOperator{\J}{J}
\DeclareMathOperator{\T}{T}
\DeclareMathOperator{\D}{D}
\DeclareMathOperator{\Null}{Null}
\DeclareMathOperator{\soc}{soc}

\DeclareMathOperator{\SNF}{SNF}
\DeclareMathOperator{\rank}{rank}
\DeclareMathOperator{\Max}{Max}

\DeclareMathOperator{\End}{End}

\DeclareMathOperator{\Ext}{Ext}

\DeclareMathOperator{\pr}{pr}

\numberwithin{equation}{section}

\newtheorem{definition}{Definition}[section]

\newtheorem{example}[definition]{Example}

\theoremstyle{remark}

\newtheorem{remark}[definition]{Remark} 

\theoremstyle{plain} 

\newtheorem{theorem}[definition]{Theorem}
\newtheorem{lemma}[definition]{Lemma}
\newtheorem{corollary}[definition]{Corollary}
\newtheorem{proposition}[definition]{Proposition}

\def\Z{\mathbb Z}
\def\C{\mathbb C}

\begin{document}

\title{On $U(\mathfrak{h})$-free modules of finite rank over $\mathfrak{sl} (2)$}

\author{Dimitar Grantcharov, Khoa Nguyen, and Kaiming Zhao}

\address{Department of Mathematics, University of Texas at Arlington, Arlington, TX 76021, USA}\email{grandim@uta.edu}

	\address{Department of Mathematics and Statistics, Queen's University, Kingston, ON K7L 3N6, Canada}
	\email{k.nguyen@queensu.ca}	
            
\address{Department of Mathematics, Wilfrid Laurier University, Waterloo, ON N2L 3C5, Canada, and School of Mathematical Science, Hebei Normal University, Shijiazhuang, Hebei 050024, China.}
\email{kzhao@wlu.ca}


\maketitle
\begin{abstract}
We study $\mathfrak{sl}(2)$-modules that are free of finite rank over
$U(\mathfrak h)$, where $\mathfrak h$ is a fixed Cartan subalgebra of $\mathfrak{sl}(2)$. These modules form a natural class of non-weight
modules. The coherent families obtained from this class via the
weighting functor are identified. We also study a distinguished class of indecomposable $U(\mathfrak h)$-free modules defined in terms of Jordan blocks and give
a recursive description of their socle filtrations. Finally, we apply the
general results to exponential modules arising from the first Weyl algebra and obtain
simplicity criteria for these modules.



\end{abstract}

\section*{Introduction}

The representation theory of semisimple Lie algebras begins with
$\mathfrak{sl}(2)$, and many general constructions and methods  are first tested in this
case. The finite-dimensional modules and, more generally, the weight modules
over $\mathfrak{sl}(2)$ are classical and explicit; we refer to~\cite{Ma} for a detailed exposition. Simple $\mathfrak{sl}(2)$-modules were
classified by Block~\cite{Bl}. However, Block's classification does not provide explicit simplicity criteria and structural results 
for many concrete classes of non-weight modules.

In this paper we study one such class: $\mathfrak{sl}(2)$-modules that are
free of finite rank over $U(\mathfrak h)=\mathbb C[h]$, where $ \mathfrak h$ is a fixed Cartan subalgebra of $\mathfrak{sl}(2)$. These modules are
Cartan-torsion-free, in contrast with weight modules, which are
$U(\mathfrak h)$-torsion.  The rank-one case is completely understood~\cite{Nil1}, while in higher rank several important results and
constructions are known; see, e.g.,~\cite{BSh,FLM, GN, GN2, GNZ,  LN, MP, Me, W}.
The $U(\mathfrak h)$-free setting requires a different approach from the one used for the weight modules. 
Our approach is to fix a  $U(\mathfrak h)$-basis and then encode   
the actions of $e$ and $f$ by polynomial matrices. A central idea of the paper is that
this matrix description can be included into a useful normal form and then used to
answer concrete representation-theoretic questions.

We first parameterize every such module  of $U(\mathfrak h)$-rank \(n\), with fixed central character,  as $M(\alpha,\boldsymbol a,K(h))$,
where  $\alpha$ determines the central character, $\boldsymbol a$ records the
invariant factors in the Smith normal form of the matrix describing the action of $f$, and
$K(h)\in\GL_n(\mathbb C[h])$.  This parameterization is used not only
for classification purposes; it is the main tool applied throughout the paper to study
weighting functors, finite-dimensional quotients, socle filtrations, duality, twisting,
and exponential examples.

Our first main result concerns the weighting functor. If $M$ is free of
finite rank over $U(\mathfrak h)$, then its weighting $\mathcal W(M)$ is a
coherent family. We compute the decomposition of
$ \mathcal W\big(M(\alpha,\boldsymbol a,K(h))\big) $
into coherent families of degree $1$. In the non-resonant cases, that is, for central characters that are not regular
integral, the answer
depends only on $\alpha$ and $\boldsymbol a$.  In the resonant
cases it also depends on a finite-dimensional linear-algebra invariant of the matrix 
$K(h)$. This gives an explicit relation between the polynomial matrix
parameter of a $U(\mathfrak h)$-free module and the decomposition of its
weighting.

As an application, we determine when
$M(\alpha,\boldsymbol a,K(h))$ has a proper submodule of the same
$U(\mathfrak h)$-rank. Equivalently, we determine when such a module has a
finite-dimensional quotient.

We also begin a systematic study of the structure of the modules $M(\alpha,\boldsymbol a, K(h))$ in 
the constant-matrix case, that is, when $K(h)$ is independent of $h$. The first
natural nontrivial indecomposable case occurs when $K(h)=\J_n$ is a Jordan
block. For the modules $M(\alpha,\boldsymbol a,\J_n)$, we compute the first
socle layer and the quotient by it. This gives a recursive description of the
socle filtration. The result separates into the non-resonant and resonant cases;
in the resonant case finite-dimensional simple modules appear in the successive
quotients.

We also discuss two functors on $U(\mathfrak h)$-free modules: the contravariant
$U(\mathfrak h)$-duality functor and the twisting functor defined by the standard involution $\tau$ of
$\mathfrak{sl}(2)$, which exchanges $e$ and $f$. These functors induce explicit transformations of the
parameters $\alpha$, $\boldsymbol a$, and $K(h)$. In particular, they
provide a systematic way to compare modules obtained from one another by duality
and $\tau$-twisting.

The final part of the paper applies the general theory to exponential modules. We consider three families
$$
T_1(g,b),\qquad T_2(g,b),\qquad T_3(g), \ \  \text{ where } g \in \mathbb C[t] \text{ and } b\in\mathbb{C},
$$
coming from explicit homomorphisms from $U(\mathfrak{sl}(2))$ to the  first Weyl
algebra. If $g$ has degree $n$, then these modules are free of rank $n$
over $U(\mathfrak h)$. We realize them explicitly as modules of the form
$M(\alpha,\boldsymbol a,K(h))$ and use this realization to prove simplicity
criteria. In particular, for $\deg g\geq2$, the modules $T_1(g,b)$ and
$T_2(g,b)$ are simple precisely when
$
2b-2\notin \mathbb Z_{\geq0}.
$ For the third family, $T_3(g)$ is simple precisely when $g$ is not an even
polynomial. If $g$ is even and nonconstant, then $T_3(g)$ decomposes into
its even and odd parts, and both summands are simple. 

The paper is organized as follows. Section~1 contains the preliminaries,
including the matrix description of $U(\mathfrak h)$-free modules and the
normal form for fixed central character. In Section~2 we compute the weighting
of $M(\alpha,\boldsymbol a,K(h))$ and its decomposition into coherent
families. Section~3 gives applications to finite-dimensional quotients and
associated short exact sequences. Section~4 studies the socle filtration of the
modules $M(\alpha,\boldsymbol a,\J_n)$. Section~5 discusses
$U(\mathfrak h)$-duality and twisting by $\tau$. Finally, Section~6 applies
the preceding results to exponential modules.

\subsection*{Acknowledgements}
The first author is partially supported by Simons Collaboration Grant 855678 and by the Bulgarian Ministry of Education and Science, Scientific Programme “Enhancing the Research Capacity in Mathematical Sciences (PIKOM)”, No. DO1-241/15.08.2023. The third author is supported by  NSERC
(311907-2026).

\section{Preliminaries}
\subsection{Notation and conventions}
Throughout this paper, $\Z$, $\C$, and $\C^*$ denote the integers, complex numbers, and nonzero complex numbers, respectively. For $k\in\Z$, let
$$\Z_{\geq k}:= \{i \in \Z: i \geq k\}.$$

 Whenever we write $z\geq w$ for complex numbers $z,w$, we mean that $z-w\in \Z_{\geq 0}$.

 Unless otherwise stated, all vector spaces, algebras, homomorphisms, and tensor products are defined over $\C$. By $\mathrm{Mat}_n(A)$ we denote the associative algebra of all $n\times n$ matrices with entries in $A$, and let $\GL_n(A)$ denote the group of invertible matrices in $\mathrm{Mat}_n(A)$. For an $A$-module $M$ and $\omega \in \Aut(A)$, $M^\omega$ stands for  the $A$-module with the same underlying vector space as $M$, but with action twisted by $\omega$, i.e. through the formula $a\cdot m^{\omega} = (\omega(a)\cdot m)^{\omega}$. 
 
 For a Lie algebra $\mathfrak{g}$, we write $U(\mathfrak{g})$ for its universal enveloping algebra.
We fix the basis $\{e, f, h\}$ of the Lie algebra $\mathfrak{sl}(2)$ to be given by the following matrices:
$$
e := \begin{pmatrix} 0 & 1\\ 0 & 0 \end{pmatrix}, \qquad 
f := \begin{pmatrix} 0 & 0\\ 1 & 0 \end{pmatrix}, \qquad 
h := \begin{pmatrix} \frac{1}{2} & 0\\ 0 & -\frac{1}{2} \end{pmatrix}.
$$

In this paper, we fix $\mathfrak h = \C h$ and  $c = (2h + 1)^2 + 4fe$. It is well known that the center   $\mathcal{Z}(\mathfrak{sl}(2))$  of  $U(\mathfrak{sl}(2))$  is $\mathbb{C}[c]$. An $\mathfrak{sl}(2)$-module $M$ is said to have central character $\gamma\in\C$ if $c$ acts on $M$ by the scalar $\gamma$. We use the same symbol $h$ both for the Cartan element of $\mathfrak{sl}(2)$ and for the variable of $\C[h] = U (\mathfrak h)$. The shift automorphism $\sigma$ of $\C[h]$ plays a crucial role in this paper; it is defined by   
$$
\sigma\big(g(h)\big):=g(h-1),
\qquad\text{for all}\;\;g(h)\in\mathbb{C}[h].
$$

Let $\mathcal D(1)$ and $\mathcal D(1)_t $ denote the algebras of differential operators on $\C[t]$ and $\C[t^{\pm1}]$, respectively. In particular, $\mathcal D(1)_t = \mathcal D(1)[t^{-1}]$ is the localization of $\mathcal D(1)$ with respect to the multiplicative subset generated by $t$.
As usual, we write $\partial=\frac{d}{dt}$. For $n \geq 1$, denote
 $$
\J_n:=
\begin{pmatrix}
1 & 1 & 0 & \cdots & 0 \\
0 & 1 & 1 & \ddots & \vdots \\
\vdots & \ddots & \ddots & \ddots & 0 \\
\vdots &  & \ddots & 1 & 1 \\
0 & \cdots & \cdots & 0 & 1
\end{pmatrix}, \qquad
\T_n:=
\begin{pmatrix}
0 & \cdots & 0 & 1\\
\vdots & \iddots & 1 & 0\\
0 & \iddots & \iddots & \vdots\\
1 & 0 & \cdots & 0
\end{pmatrix}.
$$

We denote by $\mathsf e_1,\dots,\mathsf e_n$ both the standard $\C[h]$-basis of $ \C[h]^{n} = \C[h]^{\oplus n}$ and the standard basis of $\C^n$. If $M$ is a module of finite length and $L$ is a simple module, we denote by $[M:L]$ the composition multiplicity of $L$ in $M$. If $M$ is a direct sum of indecomposable coherent families and $L$ is one of them, we denote by $[M:L]$ the multiplicity of $L$ as a direct summand of $M$. 

\subsection{The category of $U(\mathfrak h)$-free $\mathfrak{sl}(2)$-modules of finite rank}

Define $\mathcal M$ to be the full subcategory of
$U(\mathfrak{sl}(2))\text{-mod}$ whose objects are $U(\mathfrak{sl}(2))$-modules $M$ such that
$$
\operatorname{Res}^{U(\mathfrak{sl}(2))}_{U(\mathfrak h)} M
\simeq U(\mathfrak h)^{\oplus n}=\C[h]^{\oplus n}
\qquad\text{for some } n \geq 1.
$$

Throughout the paper, we identify the underlying space of any module $M\in\mathcal M$ with $\C[h]^{\oplus n}$ for some $n \geq 1$, and assume that $h$ acts by multiplication. In particular, every $\boldsymbol g\in M$ has the form
$$
\boldsymbol g=
\bigl(g_1(h)\;\;\dots\;\;g_n(h)\bigr)^{\mathsf T}
\qquad (g_j(h)\in\C[h]),
$$
and
$$
h\cdot \boldsymbol g
=
\bigl(hg_1(h)\;\;\dots\;\;hg_n(h)\bigr)^{\mathsf T}.
$$

For $n \geq 1$, define the full subcategory $\mathcal M(n)$ of $\mathcal M$ by
$$
\mathcal M(n):=
\Bigl\{\,M\in\mathcal M :
\operatorname{Res}^{U(\mathfrak{sl}(2))}_{U(\mathfrak h)} M
\simeq U(\mathfrak h)^{\oplus n}\Bigr\}.
$$

Thus, on objects,
$$
\mathcal M=\bigcup_{n\geq 1}\mathcal M(n).
$$

We will also use two larger full subcategories of $U(\mathfrak{sl}(2))\text{-mod}$. Let $\mathcal M_{\mathrm{tf}}$ denote the full subcategory of modules that are torsion-free over $U(\mathfrak h)=\C[h]$, and let $\mathcal M_{\mathrm{fg}}$ denote the full subcategory of modules that are finitely generated over $U(\mathfrak h)=\C[h]$. Since $\C[h]$ is a PID, every finitely generated torsion-free $\C[h]$-module is free of finite rank. Hence $\mathcal M$ consists precisely of the nonzero modules that belong to both $\mathcal M_{\mathrm{tf}}$ and $\mathcal M_{\mathrm{fg}}$.

In general, neither $\mathcal M_{\mathrm{tf}}$ nor $\mathcal M_{\mathrm{fg}}$ contains the other: a torsion-free $\C[h]$-module need not be finitely generated, while a finitely generated $\C[h]$-module may have $h$-torsion.

The category $\mathcal M_{\mathrm{fg}}$ is abelian, since $\C[h]$ is Noetherian. By contrast, $\mathcal M_{\mathrm{tf}}$ is not abelian: cokernels of monomorphisms between torsion-free $\C[h]$-modules may have $\C[h]$-torsion. The categories $\mathcal M$ and $\mathcal M(n)$ are not abelian either. Indeed,  as defined they do not contain the zero module, and even after adjoining it, cokernels of morphisms between free $\C[h]$-modules need not be torsion-free or free of the same rank.

\begin{lemma} \label{1st_form}
For $n \geq 1$, let $M\in\mathcal M(n)$. Then there exist
$E,F\in\Mat_n(\C[h])$ such that
\begin{equation} \label{sl2-relation}
E\sigma(F)-F\sigma^{-1}(E)=2h\I_n,
\end{equation}
and the $\mathfrak{sl}(2)$-module structure on $M$ is given by
$$
e\cdot\boldsymbol g=E\,\sigma(\boldsymbol g),
\qquad
f\cdot\boldsymbol g=F\,\sigma^{-1}(\boldsymbol g),
$$
for all $\boldsymbol g\in\C[h]^{\oplus n}$. Here $\sigma$ is applied entrywise to vectors and matrices with entries in $\C[h]$.
\end{lemma}

\begin{proof}
From $eh=(h-1)e$ and $fh=(h+1)f$, we easily prove that
$$
e\,g(h)=\sigma(g)e,\qquad f\,g(h)=\sigma^{-1}(g)f
\qquad\text{for all } g(h)\in\C[h].
$$

Define
$$
E:=\bigl(e\cdot \mathsf e_1\;\;\dots\;\;e\cdot \mathsf e_n\bigr),
\qquad
F:=\bigl(f\cdot \mathsf e_1\;\;\dots\;\;f\cdot \mathsf e_n\bigr).
$$

Then $E,F\in\Mat_n(\C[h])$, and the semilinearity gives
$$
e\cdot\boldsymbol g=E\sigma(\boldsymbol g),
\qquad
f\cdot\boldsymbol g=F\sigma^{-1}(\boldsymbol g)
$$
for all $\boldsymbol g\in\C[h]^{\oplus n}$.

Applying the relation $[e,f]=2h$ to $\boldsymbol g$ gives
$$
E\sigma(F)\boldsymbol g-F\sigma^{-1}(E)\boldsymbol g=2h\boldsymbol g.
$$

Since this holds for all $\boldsymbol g$, we obtain
$$
E\sigma(F)-F\sigma^{-1}(E)=2h\I_n.
$$
\end{proof}

\begin{definition}
For $n \geq 1$, let $E,F\in\Mat_n(\C[h])$ satisfy \eqref{sl2-relation}. We denote by $M(E,F)$ the corresponding $\mathfrak{sl}(2)$-module defined above. In particular, every object of $\mathcal M(n)$ is of the form $M(E,F)$. Throughout the paper, when we write $M(E,F)\in\mathcal M(n)$, it is implicitly understood that $E,F\in\Mat_n(\C[h])$ satisfy \eqref{sl2-relation}.
\end{definition}

The following standard change-of-basis criterion will be used repeatedly. 
\begin{proposition} \label{firstisomthm}
For $n \geq 1$, let $M(E_1,F_1),M(E_2,F_2)\in\mathcal M(n)$. Then $M(E_1,F_1)$ and $M(E_2,F_2)$ are isomorphic if and only if there exists $P(h)\in\GL_n(\C[h])$ such that
$$
E_2=P(h)^{-1}E_1P(h-1),
\qquad
F_2=P(h)^{-1}F_1P(h+1).
$$
\end{proposition}
\begin{proof}
If
$\Phi:M(E_1,F_1)\to M(E_2,F_2)$ is an isomorphism, then, with respect to the
chosen $U(\mathfrak h)$-bases, $\Phi$ is represented by a matrix
$P(h)^{-1}\in \GL_n(\mathbb C[h])$. The conditions that $\Phi$ intertwines the
actions of $e$ and $f$ are precisely the one in the statement.
\end{proof}
\begin{definition}\label{sigmadef}
Let $A,B\in\Mat_n(\C[h])$. For $\varepsilon\in\{1,-1\}$, we say that $A$ and $B$ are \emph{$\sigma^\varepsilon$-similar}, written $A\sim_{\sigma^\varepsilon}B$, if there exists $P(h)\in\GL_n(\C[h])$ such that
$$
B=P(h)^{-1}A\,\sigma^\varepsilon(P(h)).
$$

One easily verifies that $\sigma^\varepsilon$-similarity defines an equivalence relation ``$\sim_{\sigma^{\varepsilon}}$'' on $\Mat_n(\C[h])$.
\end{definition}

\subsection{Smith normal form}

We briefly recall the Smith normal form over $\C[h]$.

\begin{theorem}[Smith normal form]\label{Smith-form}
Let $A\in\Mat_n(\C[h])$. Then there exist $L,R\in\GL_n(\C[h])$ and
$d_1,\dots,d_n\in\C[h]$ such that
$$
LAR=\diag(d_1,\dots,d_n),
$$
where $d_i\mid d_{i+1}$ for $i=1,\dots,n-1$. The invariant factors
$d_i$ are uniquely determined up to multiplication by nonzero scalars in $\C$.
\end{theorem}

\begin{remark}
For $A\in\Mat_n(\C[h])$, we denote its Smith normal form by $\SNF(A)$. Throughout the paper, Smith normal forms of matrices in $\Mat_n(\C[h])$ are taken with all nonzero invariant factors to be monic. 
\end{remark}

\subsection{Modules in $\mathcal{M}(n)$ with fixed central character}
Let $M(E,F)\in \mathcal{M}(n)$ be a module with central character $\gamma\in \C$. Then, for every $\boldsymbol{g}\in \C[h]^{\oplus n}$, 
$$
\gamma\,\boldsymbol{g}
= c\cdot\boldsymbol{g}
= \bigl((2h+1)^2\I_n+4\,F\,\sigma^{-1}(E)\bigr) \boldsymbol{g}.
$$

It follows that $F\sigma^{-1}(E)=\tfrac{1}{4}\big(\gamma-(2h+1)^2\big)\I_n$. Now choose $\alpha\in\C$ such that $(2\alpha-1)^2=\gamma$, and define $\lambda_\alpha(h):=-(h-\alpha+1)(h+\alpha)$. Then 
\begin{equation}\label{EFnew}
F\sigma^{-1}(E)=\lambda_\alpha(h)\I_n.
\end{equation}

\begin{lemma} \label{1st-isom-thm}
For $n \geq 1$, let $M(E_1,F_1), M(E_2,F_2)\in \mathcal{M}(n)$ be modules with the same central character $(2\alpha-1)^2\in \C$. Then $M(E_1,F_1)\simeq M(E_2,F_2)$ if and only if $F_2 \sim_{\sigma^{-1}} F_1$.
\end{lemma}

\begin{proof}
The only-if direction follows from Proposition~\ref{firstisomthm}. Conversely, assume
$$
F_2=P^{-1}F_1\sigma^{-1}(P)
$$
for some $P\in\GL_n(\C[h])$. Since both modules have central character $(2\alpha-1)^2$,
$$
F_i\sigma^{-1}(E_i)=\lambda_\alpha(h)\I_n,\qquad i=1,2.
$$

Hence
$$
F_1\sigma^{-1}(P)\sigma^{-1}(E_2)
=
P F_2\sigma^{-1}(E_2)
=
\lambda_\alpha(h)P
=
F_1\sigma^{-1}(E_1)P.
$$

Since $F_1\sigma^{-1}(E_1)=\lambda_\alpha(h)\I_n$ with $\lambda_\alpha(h)\neq 0$, left multiplication by $F_1$ is injective on $\Mat_n(\C[h])$. Therefore
$$
\sigma^{-1}(P)\sigma^{-1}(E_2)=\sigma^{-1}(E_1)P.
$$

Applying $\sigma$ on both sides leads to 
$$
E_2=P^{-1}E_1P(h-1).
$$

Thus Proposition~\ref{firstisomthm} implies $M(E_1,F_1)\simeq M(E_2,F_2)$.
\end{proof}

\begin{definition} \label{tripledef}
Let $\alpha \in \C$, $\boldsymbol{a} = (a_-, a_0, a_+)$ be a triple with $a_-, a_+ \in \mathbb{Z}_{\geq 0}$ and $a_0 \in \mathbb{Z}$ such that $a_- + |a_0| + a_+ = n$.  We define the diagonal matrix $ P_{(\boldsymbol{a},\alpha)}(x) \in \Mat_n(\mathbb{C}[x]) $ as follows:
\begin{itemize}
    \item[(i)] $ P_{(\boldsymbol{a},\alpha)}(x)_{ii} = 1 $ for $ i = 1, \ldots, a_- $,
    \item[(ii)] If $ a_0 \geq 0 $, then $ P_{(\boldsymbol{a},\alpha)}(x)_{ii} = x - \alpha+1 $ for $ i = a_- + 1, \ldots, a_- + a_0 $,
    \item[(iii)] If $ a_0 < 0 $, then $ P_{(\boldsymbol{a},\alpha)}(x)_{ii} = x + \alpha  $ for $ i = a_- + 1, \ldots, a_- - a_0 $,
    \item[(iv)] $ P_{(\boldsymbol{a},\alpha)}(x)_{ii} = (x - \alpha+1)(x + \alpha) $ for $ i = a_- + |a_0| + 1, \ldots, n $.
\end{itemize}
We also denote $\overline{P}_{(\boldsymbol{a},\alpha)}(x)  \in \Mat_n(\mathbb C[x]) $ so that 
$\overline{P}_{(\boldsymbol{a},\alpha)}(x)  P_{(\boldsymbol{a},\alpha)}(x)  = \lambda_\alpha(x) \I_n.$
\end{definition}

\begin{theorem} \label{thm-h-free}
For $n \geq 1$, let $M\in\mathcal M(n)$ be a module admitting the central character $(2\alpha -1)^2$. Then there exist $K(h)\in\GL_n(\C[h])$ and a triple 
$$\boldsymbol{a} = (a_-, a_0, a_+) \in \Z_{\geq 0} \times \Z \times \Z_{\geq 0}$$ 
satisfying $a_- + |a_0| + a_+ = n$ such that
$$M\simeq M\Big(\sigma\left(K^{-1}(h)\overline{P}_{(\boldsymbol{a},\alpha)}(h)\right),P_{(\boldsymbol{a},\alpha)}(h) K(h)\Big).$$
\end{theorem}

\begin{proof}
By Lemma~\ref{1st_form}, there exist matrices $E$ and $F$ such that 
$$M = M(E,F),\qquad  F\sigma^{-1}(E) = \lambda_\alpha(h)\I_n.$$

Applying Theorem~\ref{Smith-form} to $F$, there exist $L(h), R(h)\in \GL_n(\C[h])$ such that
$$
L(h)FR(h) =\SNF(F) :=  \diag(\mu_1,\dots,\mu_n).
$$

Notice that
\begin{equation}\label{equaA}
\lambda_\alpha(h)\I_n = L(h)\bigl(F\,\sigma^{-1}(E)\bigr)L(h)^{-1}= \diag(\mu_1,\ldots,\mu_n)\;R(h)^{-1}\,\sigma^{-1}(E)\;L(h)^{-1}.
\end{equation}

It follows that $R(h)^{-1}\sigma^{-1}(E)L(h)^{-1}=\diag(\xi_1,\ldots,\xi_n),$ with $\xi_i\in\C[h]$ satisfying \\ $
\mu_i\,\xi_i=\lambda_\alpha(h)$. Hence, there exist $\boldsymbol{a} = (a_-, a_0, a_+)\in \Z_{\geq 0} \times \Z \times \Z_{\geq 0}$ such that
$$\diag(\mu_1,\dots,\mu_n) = P_{(\boldsymbol{a},\alpha)}(h), \qquad \diag(\xi_1,\ldots,\xi_n) = \overline{P}_{(\boldsymbol{a},\alpha)}(h).$$

Moreover, \eqref{equaA} can be rewritten as
$$
\lambda_\alpha(h)\I_n
=L(h)FL(h+1)^{-1}\,L(h+1)R(h)\,
\sigma^{-1}\Bigl(R(h-1)^{-1}L(h)^{-1}\,L(h)EL(h-1)^{-1}\Bigr).
$$

Let $K(h):=R(h)^{-1}L(h+1)^{-1}\in\GL_n(\C[h])$. Then
$$L(h)FL(h+1)^{-1} = P_{(\boldsymbol{a},\alpha)}(h)K(h), \qquad L(h)EL(h-1)^{-1} = \sigma\left(K^{-1}(h)\overline{P}_{(\boldsymbol{a},\alpha)}(h)\right).$$

By Proposition~\ref{firstisomthm}, we conclude that
$$M(E,F)\simeq M\Big(\sigma\left(K^{-1}(h)\overline{P}_{(\boldsymbol{a},\alpha)}(h)\right),P_{(\boldsymbol{a},\alpha)}(h) K(h)\Big),$$
and the theorem follows.
\end{proof}

\begin{definition}
We use the following shorthand notation for the above module.
$$M\big(\alpha,\boldsymbol{a},K(h)\big):= M\Big(\sigma\left(K^{-1}(h)\overline{P}_{(\boldsymbol{a},\alpha)}(h)\right),P_{(\boldsymbol{a},\alpha)}(h) K(h)\Big).$$
\end{definition}

\begin{corollary}\label{howtofind-K(h)}
For $n \geq 1$, let $M(E,F) \in \mathcal M(n)$ be a module admitting the central character $(2\alpha -1)^2$. If $L(h), R(h)\in\GL_n(\C[h])$ satisfy $L(h) F R(h)=P_{(\boldsymbol{a},\alpha)}(h)$, then 
$$M(E,F) \simeq  M\big(\alpha,\boldsymbol{a},R(h)^{-1}L(h+1)^{-1}\big).$$
\end{corollary}

\begin{theorem}\label{main-isom-thm}
Let $\alpha,\beta\in\C$, let $\boldsymbol{a}=(a_-,a_0,a_+)$ and $\boldsymbol{b}=(b_-,b_0,b_+)$ be elements of $\Z_{\geq0}\times\Z\times\Z_{\geq0}$, and let $K_1(h),K_2(h)\in\GL_n(\C[h])$. Then
$$
M\big( \alpha, {{\boldsymbol a}}, K_1(h) \big) \simeq M\big( \beta,{\boldsymbol b}, K_2(h) \big)
$$
if and only if the following conditions hold:
\begin{itemize}
\item[(i)] either $\beta=\alpha$ and $\boldsymbol b=\boldsymbol a$, or $\beta=1-\alpha$ and $\boldsymbol b=(a_-,-a_0,a_+)$;
\item[(ii)] there exists $P(h)\in\GL_n(\C[h])$ such that
$$
P(h)\,P_{(\boldsymbol a,\alpha)}(h)\,K_2(h)\;=\;P_{(\boldsymbol a,\alpha)}(h)\,K_1(h)\,P(h+1).
$$
\end{itemize}
\end{theorem}

\begin{proof}
We first prove the ``if" direction of the statement. Condition (i) implies that $ P_{({\boldsymbol b}, \beta)}(h) = P_{({\boldsymbol a}, \alpha)}(h)$. Combining conditions (i) and (ii), it follows that there exists $P(h)\in\GL_n(\C[h])$ such that
$$
P_{({\boldsymbol b},\beta)}(h)\,K_2(h)\;=\; P(h)^{-1}\,P_{({\boldsymbol a},\alpha)}(h)\,K_1(h)\,P(h+1).
$$

Hence, Lemma \ref{1st-isom-thm} implies
$$
M\big( \alpha, {{\boldsymbol a}}, K_1(h) \big) \simeq M\big( \beta,{\boldsymbol b}, K_2(h) \big).
$$

For the ``only if" direction, assume $M\big( \alpha, {{\boldsymbol a}}, K_1(h) \big) \simeq M\big( \beta,{\boldsymbol b}, K_2(h) \big)$. Then they have the same central character; i.e., $ (2\alpha -1)^2 = (2\beta -1)^2$. Therefore, $\alpha=\beta$ or $\alpha=1-\beta$. By Lemma~\ref{1st-isom-thm} there exists $P(h)\in \GL_n(\C[h])$ such that
\begin{equation*}
P(h)^{-1}\,P_{({\boldsymbol a},\alpha)}(h)\,K_1(h)\,P(h+1)\;=\;P_{({\boldsymbol b},\beta)}(h)\,K_2(h).
\end{equation*}

Since $P(h)^{-1}, K_1(h)P(h+1)K_2(h)^{-1}\in \GL_n(\C[h])$,
$$P_{({\boldsymbol a},\alpha)}(h) = \SNF\big(P_{({\boldsymbol a},\alpha)}(h)\big) =  P_{({\boldsymbol b},\beta)}(h).$$

Hence, $\boldsymbol a=\boldsymbol b$ if $\alpha=\beta$; $\boldsymbol b=(a_-,-a_0,a_+)$ if $\alpha=1-\beta$; and
$$
P(h)\,P_{({\boldsymbol a},\alpha)}(h)\,K_2(h)
= P_{({\boldsymbol a},\alpha)}(h)\,K_1(h)\,P(h+1).
$$
\end{proof}

\subsection{Some examples of modules in $\mathcal M(n)$}
We first recall the classification of the rank-one $U(\mathfrak h)$-free modules.
\begin{theorem}[{\cite[Theorem 9]{Nil1}}] \label{classfication-rank1}
Every object in $\mathcal{M}(1)$ is isomorphic to exactly one of the following modules:
$$
M\bigl(\alpha, (1,0,0),\beta\bigr),\quad
M\bigl(\alpha, (0,1,0),\beta\bigr)\ \ \big( \alpha\in\C_{\geq \frac{1}{2}}\big),\quad
M\bigl(\alpha, (0,0,1), \beta\bigr),
$$
with $\alpha\in\C$, $\beta\in\C^*$, and 
$\C_{\geq \tfrac12}:=\big\{\,z\in\C:\re(z)\geq \tfrac{1}{2}\,\big\}$.
\end{theorem}

\begin{proposition}[{\cite[Lemmas 11 and 12]{Nil1}}]\label{simplicity-rank1}
Let $\alpha \in \C$ and $\beta\in\C^{*}$. Then
\begin{itemize}
\item[(i)] The modules $M\big(\alpha, (1,0,0), \beta\big)$ and $M\big(\alpha, (0,0,1), \beta\big)$ are simple.
\item[(ii)] If $ \alpha\in \C_{\geq \tfrac{1}{2}} \setminus\{\tfrac{1}{2}\Z_{\geq 0}+1\}$ then the module $M\big(\alpha, (0,1,0), \beta\big)$ is simple. If $\alpha\in\frac{1}{2}\Z_{\geq 0}+1$, then the following short exact sequence is non-split
      $$
      0 \longrightarrow M\big(\alpha, (0,-1,0), \beta\big) \longrightarrow M\big(\alpha, (0,1,0), \beta\big)\longrightarrow
      L(2\alpha-2)\longrightarrow0.
      $$
\end{itemize}
\end{proposition}

\begin{example}
In \cite{BSh}, Bahturin and Shihadeh constructed a one-parameter family of simple ${U}(\mathfrak h)$-free modules of rank $2$. For each $\omega\in\C$, the module $M^C_{\omega}$ may be realized explicitly as follows.
$$M^C_{\omega} \simeq M\left(\begin{pmatrix}
0 & \frac{1}{2}\left(-4h^{2}+4h+\omega^{2}+2\omega\right)\\
\frac{1}{2} & 0
\end{pmatrix},  \begin{pmatrix}
0 & \frac{1}{2}\left(-4h^{2}-4h+\omega^{2}+2\omega\right)\\
\frac{1}{2} & 0
\end{pmatrix}\right).$$

Since
$$\begin{pmatrix}0&1\\ -1&0\end{pmatrix}
\begin{pmatrix}
0 & \frac{1}{2}\left(-4h^{2}-4h+\omega^{2}+2\omega\right)\\
\frac{1}{2} & 0
\end{pmatrix}
\begin{pmatrix}2&0\\0&\tfrac12\end{pmatrix}
=P_{\left((1,0,1),\;\tfrac{\omega+2}{2}\right)}(h),$$
then by Corollary \ref{howtofind-K(h)},
$$
M^C_{\omega}\simeq M\left(\tfrac{\omega+2}{2}, (1,0,1), \begin{pmatrix}0&-\frac{1}{2}\\ 2&0\end{pmatrix}
\right).
$$
\end{example}

\begin{example}
In \cite{MP}, Martin and Prieto constructed a family of finite-rank $U(\mathfrak h)$-free modules parametrized by $\beta\in\C^*$, $\mu\in\C$, and a nonconstant polynomial $p(h)\in\C[h]$. For these modules, the $f$-matrix is given by
$$F(h) = (-h+\mu)(-h-\mu-1) K(h),$$
where 
$$K(h) := \begin{pmatrix}
0 & 1 & 0 & \cdots & 0 & 0 \\
0 & 0 & 1 & \ddots & \vdots & \vdots \\
\vdots & \vdots & \ddots & \ddots & 0 & 0 \\
0 & 0 & \cdots & 0 & 1 & 0 \\
-\tfrac{p(-h)}{\beta} & 0 & \cdots & 0 & 0 & 1 \\
\tfrac{1}{\beta} & 0 & \cdots & 0 & 0 & 0
\end{pmatrix}.$$

Therefore, the module constructed in \cite{MP} is isomorphic to $M\left(\mu +1, (0,0,n), K(h)
\right)$.
\end{example}

\subsection{Weighting functor and coherent families of $\mathfrak{sl}(2)$}

The weighting functor $\mathcal W$ was first introduced in \cite{Nil1}. In this subsection, we recall its definition in the case of $\mathfrak{sl}(2)$.

Let $M$ be a $U(\mathfrak{sl}(2))$-module. For $\lambda\in\C$, let
$\overline{\lambda}:U(\mathfrak h)\to\C$ be the algebra homomorphism defined by
$\overline{\lambda}(h)=\lambda$. The \emph{weighting} of $M$ is defined by
$$
\mathcal W(M)
:=
\bigoplus_{\mathfrak m\in\Max(\C[h])} M/\mathfrak m M
=
\bigoplus_{\lambda\in\C} M/\ker\overline{\lambda}\,M,
$$
where $\Max(\C[h])$ denotes the set of maximal ideals of $\C[h]$. The space
$\mathcal W(M)$ becomes a $U(\mathfrak{sl}(2))$-module under the action
\begin{align*}
h\cdot (v+\ker\overline{\lambda}\,M)
&:=
h v+\ker\overline{\lambda}\,M,\\
e\cdot (v+\ker\overline{\lambda}\,M)
&:=
e v+\ker\overline{\lambda+1}\,M,\\
f\cdot (v+\ker\overline{\lambda}\,M)
&:=
f v+\ker\overline{\lambda-1}\,M.
\end{align*}

\begin{definition} \label{cohfam-def}
An $\mathfrak{sl}(2)$-coherent family of degree $d\geq 1$ is a weight module
$$
M=\bigoplus_{\lambda\in\C} M_\lambda, \qquad M_\lambda = \{ m \in M \; | \; hm = \lambda m\},
$$
satisfying the following conditions:
\begin{itemize}
\item[(i)] $\dim M_\lambda=d$ for every $\lambda\in\C$;
\item[(ii)] the function $\lambda\mapsto \Tr\bigl(c|_{M_\lambda}\bigr)$ is polynomial in $\lambda$.
\end{itemize}
\end{definition}

\begin{remark}
For $n\geq 1$ and $M\in\mathcal M(n)$, it is well known from \cite{Nil1} that $\mathcal W(M)$ is a coherent family of degree $n$.
\end{remark}

\section{Coherent families and $U(\mathfrak h)$-free modules}\label{coherentfam-section}
In this section, we study the weighting of $M\big(\alpha,\boldsymbol a,K(h)\big)$, which is a coherent family of degree $n=a_-+|a_0|+a_+$. In particular, we provide an alternative realization of $\mathcal W\big(M(\alpha,\boldsymbol a,K(h))\big)$ in terms of differential operators and describe its decomposition as a direct sum of coherent families of degree $1$.
\subsection{Modules over ${\mathcal D}(1)$ and  $\mathcal D(1)_t $} 
We introduce the following two important modules of $\mathcal D(1)_t $.
\begin{definition}
For $\lambda \in \C$, let $\mathcal F[\lambda] := t^{\lambda} \C [t^{\pm 1}]$ regarded as a $\mathcal D(1)_t $-module with the natural action of $t$, $t^{-1}$, and $\partial$. For $\overline{\lambda} = \lambda + \Z \in \C/\Z$, we write $\mathcal F\left[\overline{\lambda}\right]  = \mathcal F[\lambda] $, and set
$$
\mathcal F := \bigoplus_{\overline{\lambda} \in \C/\Z} \mathcal F\left[\overline{\lambda}\right]. 
$$

Denote by $\varphi : \mathcal D(1)_t  \to \End(\mathcal F)$ the corresponding representation on $\mathcal F$.
\end{definition}

\begin{definition}
For $\eta \in \C^*$, let $G_\eta$ be the $\mathcal D(1)_t$-module with underlying space $\C[h]$, whose action is given by  $t \mapsto \eta\sigma$, $\partial \mapsto \frac{1}{\eta} (h+1)\sigma^{-1}$. Denote by $\beta : \mathcal D(1)_t  \to \End(\mathbb C [h])$
the homomorphism corresponding to the $\mathcal D(1)_t $-module  $G_1$.
\end{definition}

\begin{remark}
 For $\eta\in\mathbb C^*$, set
$$
\overline G_\eta := \Res^{\mathcal D(1)_t}_{\mathcal D(1)} G_\eta .
$$

One can easily verify that
$$\overline G_\eta \simeq E(\eta t)^\theta,$$
where $E(g(t))$ denotes the natural ${\mathcal D}(1)$-module on $e^{g(t)}\C[t]$ and $\theta$ is the Fourier-transform automorphism of $\mathcal D(1)$, given by $t\mapsto \partial,\;  \partial\mapsto -t$. Moreover, every $\mathcal D(1)_t $-module that is free of rank $1$ over $\mathbb C[t\partial]$ is isomorphic to $G_\eta$ for some $\eta \in \C^*$. Every weight $\mathcal D(1)_t $-module  $M$ of degree $1$ is a direct sum of modules of the form $\mathcal F\left[\overline{\lambda}\right]$, where $\overline{\lambda}\in \C/\mathbb Z$. In particular, if the support of $M$ is $\C$, then $M$ is isomorphic to $\mathcal F$. Finally, every ${\mathcal D}(1)$-module that is free of rank $1$ over $\mathbb C[t\partial]$ is isomorphic to
either $\overline G_\eta$ or $\overline G_\eta^\theta$, for some $\eta \in \C^*$.
\end{remark}

\begin{remark}
Rank-one ${U}(\mathfrak h)$-free modules over generalized Weyl algebras, in particular over $\mathcal D(1)_t$ and ${\mathcal D}(1)$, were studied in detail in \cite{LN}.
\end{remark}

\begin{proposition}
Let $\alpha \in\C$, $\boldsymbol a=(a_-,a_0,a_+)\in \Z_{\geq 0}^3$ so that $a_- + a_0 + a_+ =n \geq 1$, and let $K(x)\in\GL_n(\C[x])$. Then the correspondence
\begin{eqnarray*}
f &\mapsto &P_{(\boldsymbol a,\alpha)}(t \partial) K(t\partial) t^{-1} \I_n, \\
e  &\mapsto &  tK^{-1}(t\partial)\overline{P}_{(\boldsymbol a,\alpha)}(t\partial),\\
h &\mapsto& t\partial \I_n,
\end{eqnarray*}
extends to a homomorphism of associative algebras $\omega_{\alpha,\boldsymbol a, K} : U({\mathfrak{sl} (2)}) \to  \Mat_n(\mathcal D(1)_t )$.
\end{proposition}

\begin{proof}
Direct verification.
\end{proof}

\begin{corollary}
The module $M\big(\alpha,\boldsymbol a,K(h)\big)$ is isomorphic to the module with underlying space $\C[h]^{\oplus n}$ and corresponding homomorphism $\Mat_n(\beta) \omega_{\alpha,\boldsymbol a, K}$, see the diagram below.

\begin{center}
\begin{tikzcd}
U(\mathfrak{sl}(2)) \arrow[r, "\omega_{\alpha,\boldsymbol a, K}"] \arrow[d] 
& \Mat_n(\mathcal D(1)_t ) \arrow[d, "\Mat_n(\beta)"] \\
\End(\mathbb C[h]^{\oplus n}) \arrow[r, "\simeq"]
&  \Mat_n(\End(\mathbb C[h]))
\end{tikzcd}
\end{center}

\end{corollary}

\begin{definition}
Define an $\mathfrak{sl}(2)$-module structure on $\mathcal F^{\oplus n}$ via the composition
$$
\begin{tikzcd}[column sep=large]
U(\mathfrak{sl}(2)) \arrow[r, "\omega_{\alpha,\boldsymbol a,K}"] & \Mat_n(\mathcal D(1)_t ) \arrow[r, "\Mat_n(\varphi)"] & \Mat_n(\End(\mathcal F)) \simeq \End(\mathcal F^{\oplus n}),
\end{tikzcd}
$$
and denote the resulting module by $\mathcal F(\alpha, \boldsymbol a, K(h))$.
\end{definition}

\begin{proposition}
Let $\alpha \in\C$, $\boldsymbol a=(a_-,a_0,a_+)\in \Z_{\geq 0}^3$ so that $a_- + a_0 + a_+ =n \geq 1$, and let $K(h)\in\GL_n(\C[h])$. Then
$$\mathcal{W}\left(M\big(\alpha,{\boldsymbol a},K(h)\big)\right) \simeq \mathcal F(\alpha, \boldsymbol a, K(h)).$$
\end{proposition}

\begin{proof}
For $i \in \{1,\dots, n\}$ and $\lambda \in \C$, denote 
$$v_{i,\lambda}:= e_i + \ker\overline{\lambda}\; M\big(\alpha,{\boldsymbol a},K(h)\big).$$

Then, 
$$\mathcal{W}\left(M\big(\alpha,{\boldsymbol a},K(h)\big)\right) = \Span\{v_{i,\lambda}:1 \leq i \leq n, \lambda \in \C  \}.$$

The $\mathfrak{sl}(2)$-action on the module $\mathcal{W}\left(M\big(\alpha,{\boldsymbol a},K(h)\big)\right)$ can be given in matrix form as follows. For each $\lambda\in\mathbb C$, set $\mathbf v_\lambda:=\bigl(v_{1,\lambda}\;\;\dots\;\;v_{n,\lambda}\bigr)
\in
\Mat_{1\times n}\!\left(\mathcal W\left(M(\alpha,\boldsymbol a,K(h))\right)\right).$
Then
$$h\cdot \mathbf v_\lambda=\lambda \mathbf v_\lambda, \quad f\cdot \mathbf v_\lambda 
=
\mathbf v_{\lambda-1}
P_{(\boldsymbol a,\alpha)}(\lambda-1)K(\lambda-1), \quad
e\cdot \mathbf v_\lambda
=
\mathbf v_{\lambda+1}
K^{-1}(\lambda)\overline P_{(\boldsymbol a,\alpha)}(\lambda).
$$

The isomorphism is given by
$$ \Upsilon : \mathcal{W}\left(M\big(\alpha,{\boldsymbol a},K(h)\big)\right) \to \mathcal F(\alpha, \boldsymbol a, K(h)),\qquad \text{where}\qquad v_{i, \lambda} \mapsto t^\lambda e_i,$$
for all $i \in \{1,\dots,n\}$ and $\lambda \in \C$.
\end{proof}


\subsection{Four coherent families}

 We define four coherent families of degree $1$ that will play crucial role in this section.
\begin{definition} For $\alpha \in \C$, let $\mathcal F_-(\alpha)$, $\mathcal F_0^-(\alpha)$, $\mathcal F_0^+(\alpha)$, and $\mathcal F_+(\alpha)$ denote the $\mathfrak{sl}(2)$-modules with underlying space $\Span \left\{ v_{\lambda} : \lambda \in \mathbb C\right\}$ as follows.
\begin{itemize}
\item[(i)] The family $\mathcal F_-(\alpha)$:
$$e\cdot v_\lambda =-(\lambda-\alpha+1)(\lambda+\alpha)v_{\lambda+1}, \qquad f\cdot v_\lambda = v_{\lambda-1}, \qquad h\cdot v_{\lambda}= \lambda v_\lambda. $$
\item[(ii)] The family $\mathcal F_0^-(\alpha)$:
$$e\cdot v_{\lambda}= -(\lambda -\alpha +1) v_{\lambda +1}, \qquad f\cdot v_{\lambda} = (\lambda +\alpha-1) v_{\lambda-1}, \qquad h \cdot v_\lambda = \lambda v_\lambda. $$
\item[(iii)] The family $\mathcal F_0^+(\alpha)$:
$$e\cdot v_{\lambda}= -(\lambda +\alpha) v_{\lambda +1}, \qquad f\cdot v_{\lambda} = (\lambda -\alpha) v_{\lambda-1}, \qquad h \cdot v_\lambda = \lambda v_\lambda. $$
\item[(iv)] The family $\mathcal F_+(\alpha)$:
$$e\cdot v_{\lambda}= -v_{\lambda +1},\qquad f\cdot v_{\lambda}=(\lambda-\alpha)(\lambda+\alpha-1)v_{\lambda-1}, \qquad h\cdot v_\lambda = \lambda v_\lambda.$$
\end{itemize}

Note that $\mathcal F_-(\alpha)^{\vee} \simeq \mathcal F_+(\alpha)$ and $\mathcal F_0^-(\alpha)^{\vee} \simeq \mathcal F_0^+(\alpha)$, where $(-)^{\vee}$ denotes the restricted dual. Also, $\mathcal F_0^-\left(\frac{1}{2}\right) = \mathcal F_0^+ \left(\frac{1}{2}\right)$ in the singular central character case $\alpha=\frac{1}{2}$. 
\end{definition}

One can show that every coherent family of degree $1$ with the property that $\Null e + \Null f = 2$ is isomorphic to one and exactly one of the four families above. Also, the only coherent families that have a finite-dimensional quotient are 
$$ \mathcal F_0^-(\alpha)\;\;\; \text{for}\;\;\; \alpha \in \tfrac{1}{2}\Z_{\leq 0},\quad \mathcal F_0^+(\alpha)\;\;\; \text{for}\;\;\;  \alpha \in 1+ \tfrac{1}{2}\Z_{\geq 0}.$$

\subsection{Decomposition of $\mathcal{W}\left(M\big(\alpha,{\boldsymbol a},K(h)\big)\right)$} 
\begin{lemma} \label{famofmatrices}
Let $M$ and $N$ be $\mathfrak{sl}(2)$-coherent families of degree $d \geq 1$. Suppose that the actions of $e$ and $f$ on $M$ and $N$ are given by
$$E_M(\lambda) : M_\lambda \to M_{\lambda +1}, \quad F_M(\lambda): M_\lambda \to M_{\lambda-1};\quad E_N(\lambda) : N_\lambda \to N_{\lambda +1}, \quad F_N(\lambda): N_\lambda \to N_{\lambda-1}.$$

Then $M\simeq N$ if and only if there exists a family $\big\{A(\lambda) \in \GL_d(\C): \lambda \in \C \big\}$ such that
$$A(\lambda +1)E_M(\lambda) = E_N(\lambda)A(\lambda), \quad  A(\lambda -1)F_M(\lambda) = F_N(\lambda)A(\lambda),\quad \text{for all}\;\;\; \lambda \in \C.$$
\end{lemma}
\begin{proof}
An isomorphism $M\to N$ preserves weight spaces, and its restriction to
$M_\lambda$ is represented by some $A(\lambda)\in\GL_d(\mathbb C)$.
Commuting with $e$ and $f$ gives exactly the two displayed identities.
Conversely, a family $A(\lambda)$ satisfying these identities defines a
weight-preserving linear isomorphism $M\to N$ which commutes with $e$, $f$,
and hence with the $\mathfrak{sl}(2)$-action.
\end{proof}

Let $\boldsymbol a=(a_-,a_0,a_+) \in \Z_{\geq 0}^3$ satisfy $a_-+a_0+a_+=n \geq 1$ and $K(h) \in \GL_n(\C[h])$. For  $\alpha \in \tfrac{1}{2}\Z \setminus \{\tfrac{1}{2}\}$, define
$$\mathcal K_{\alpha, \boldsymbol a, K}:=\begin{cases} K(-\alpha) \prod_{j=1}^{2\alpha -2}\left(P_{({\boldsymbol{a}},\alpha)}(-\alpha+j) K(-\alpha+j)\right), &\alpha \geq \tfrac{3}{2},\\ K(-\alpha), & \alpha=1,\\K(\alpha-1), & \alpha=0, \\ K(\alpha-1) \prod_{j=1}^{-2\alpha}\left(P_{({\boldsymbol{a}},\alpha)}(\alpha-1+j) K(\alpha-1+j)\right), &\alpha \leq -\tfrac{1}{2}.  \end{cases}$$

For $T \in \Mat_n(\C)$, define $\boldsymbol M_{\boldsymbol a}(T) \in \Mat_{n \times (a_- +a_+)}(\C)$ to be the following matrix
$$\boldsymbol M_{\boldsymbol a}(T):= \Big( T \mathsf e_1\;\;\dots\;\; T \mathsf e_{a_-}\;\; \mathsf e_{a_-+a_0+1}\;\;\dots\;\;\mathsf e_n\Big).$$

For $0\leq d \leq \min\{a_-, a_+ \}$, denote
$$\mathcal S^- := \Span\{\mathsf e_1,\dots,\mathsf e_{a_-}\}, \qquad \mathcal S^+ := \Span\{\mathsf e_{a_-+a_0+1},\dots, \mathsf e_n\},  $$
$$\mathcal S^-_{a_- -d} := \Span\{\mathsf e_1,\dots,\mathsf e_{a_- -d}\}, \qquad \mathcal S^+_d := \Span\{\mathsf e_{a_-+a_0+1},\dots, \mathsf e_{a_-+a_0+d}\}. $$

\begin{theorem} \label{cohfamilydecomp}
The decomposition of $\mathcal{W}\left(M\big(\alpha, \boldsymbol a, K(h)\big)\right)$ into a direct sum of coherent families of degree $1$ depends on $\alpha$, $\boldsymbol a$, and $K(h)$ as described below.
\begin{itemize}
\item[(i)] Assume that  either $(a_-,a_+) \in \{(0,0), (m,0), (0,m): m \geq 1 \}$; \\ or $(a_-,a_+) \in \Z_{\geq 1} \times \Z_{\geq 1}$ with 
$\alpha \in \left(\C \setminus \tfrac{1}{2}\Z\right) \cup \{\tfrac{1}{2}\}.$ Then for every $K(h) \in \GL_n(\C[h])$,
$$\mathcal{W}\left(M\big(\alpha, \boldsymbol a, K(h)\big)\right) \simeq\mathcal F_-(\alpha)^{\oplus  a_-} \oplus \mathcal F_0^+(\alpha)^{\oplus a_0} \oplus \mathcal F_+(\alpha)^{\oplus  a_+}.$$ 
\item[(ii)] If $(a_-,a_+) \in \Z_{\geq 1} \times \Z_{\geq 1}$ and $\alpha \in \tfrac{1}{2}\Z \setminus \{\tfrac{1}{2}\}$, then  
$$\mathcal{W}\left(M\big(\alpha, \boldsymbol a, K(h)\big)\right)\simeq\mathcal F_-(\alpha)^{\oplus  (a_- - d)} \oplus \mathcal F_0^+(\alpha)^{\oplus (a_0+d)} \oplus \mathcal F_0^-(\alpha)^{\oplus d} \oplus \mathcal F_+(\alpha)^{\oplus  (a_+ -d)},$$
where $d:= a_-+ a_+ -\rank\left(\boldsymbol M_{\boldsymbol a}\big(\mathcal K_{\alpha, \boldsymbol a, K} \big)\right)$. 
\end{itemize}

\end{theorem}

\begin{remark} \label{d-intersectiondim}
Note that $d = \dim\big(\mathcal K_{\alpha, \boldsymbol a, K} \mathcal S^- \cap \mathcal S^+ \big)$. Therefore, $0\leq d \leq \min\{a_-, a_+ \}$.
\end{remark}

\begin{proof}
By Lemma \ref{famofmatrices}, to prove part (i), it suffices to show that there exists a family 
$$\big\{A(\lambda) \in \GL_n(\C): \lambda \in \C \big\}$$ 
such that
\begin{itemize}
\item[(1)] $A(\lambda)P_{({\boldsymbol{a}},\alpha)}(\lambda) K(\lambda) =P_{({\boldsymbol{a}},\alpha)}(\lambda) A(\lambda +1);$
\item[(2)]$   A(\lambda +1) K^{-1}(\lambda) \overline{P}_{({\boldsymbol{a}},\alpha)}(\lambda)  = \overline{P}_{({\boldsymbol{a}},\alpha)}(\lambda)   A(\lambda).$
\end{itemize}

Let $\Theta$ be a set of representatives for the cosets of $\C/\Z$, chosen so that if $\alpha -1$ and $-\alpha$ lie in distinct cosets, then $\{\alpha-1, -\alpha\} \subset \Theta$, whereas if they lie in the same coset, then $-\alpha \in \Theta$. Let $p(\lambda) = (\lambda-\alpha+1)(\lambda+\alpha)$.

If the two roots $\alpha-1$ and $-\alpha$ lie in distinct
$\mathbb Z$-cosets, or if they coincide, i.e. $\alpha=\frac12$, we define
$$
A(\alpha -1) = A(-\alpha) = \I_n, \qquad
A(\alpha)=K(\alpha-1),\qquad A(-\alpha+1)=K(-\alpha).
$$

In these cases the assignments are compatible, and the two equations above
hold at $\lambda=\alpha-1$ and $\lambda=-\alpha$.  For all other
$\lambda\in(\alpha-1)+\mathbb Z$ or $\lambda\in-\alpha+\mathbb Z$, we
define $A(\lambda)$ recursively by
$$
A(\lambda+1)
=
P_{(\boldsymbol a,\alpha)}(\lambda)^{-1}
A(\lambda)P_{(\boldsymbol a,\alpha)}(\lambda)K(\lambda),
\qquad
\lambda\notin\{\alpha-1,-\alpha\}.
$$

It remains to consider the case where $\alpha-1$ and $-\alpha$ are
distinct but lie in the same $\mathbb Z$-coset.  In part (i), this can occur
only in the one-sided cases.  Let  $a_+=0$ (the case $a_-=0$ is handled in a similar manner). Set
$$
A(\alpha-1)=\I_n,\qquad A(\alpha)=K(\alpha-1),
$$
and use equation $(1)$ to define $A(\lambda)$ recursively for all other
$\lambda$ in this coset.  Since $P_{(\boldsymbol a,\alpha)}(\lambda)$ is
invertible except at $\lambda=\alpha-1$, and
$\overline P_{(\boldsymbol a,\alpha)}(-\alpha)=0$ when $a_+=0$, equation
$(2)$ also holds.

 For each coset $\nu+\Z$ with $\nu\in \Theta\setminus\{\alpha-1, -\alpha\}$, we first choose an arbitrary matrix
$$A(\nu) \in \GL_n(\C),$$
for instance, $A(\nu)=\I_n$. The remaining matrices $A(\lambda)$, for $\lambda\in \nu+\Z$, are then determined inductively using equation (1), as above. This proves part (i).

\medskip\noindent We prove part (ii) for the case $\alpha \in 1+\tfrac{1}{2}\Z_{\geq 0}$; the case $\alpha \in \tfrac{1}{2}\Z_{\leq 0}$ follows by an analogous argument. Assume that $2\alpha -1=N \geq 1$, then the two roots of $p(\lambda)$ lie in the coset $\alpha+ \Z$ ($-\alpha+N = \alpha-1$). For $j \in \Z$, let $\lambda_j:= -\alpha+j$. Then $\lambda_0 = -\alpha$ and $\lambda_N=\alpha-1$ are two roots of $p(\lambda)$. For brevity, let
$$D_j :=  \diag \big(\underbrace{1,\dots,1}_{(a_- -d)\ \text{times}},\underbrace{\lambda_j-\alpha+1,\dots,\lambda_j-\alpha+1}_{(a_0+d)\ \text{times}},\underbrace{\lambda_j+\alpha,\dots,\lambda_j+\alpha}_{d\ \text{times}}, \underbrace{p(\lambda_j),\dots,p(\lambda_j)}_{(a_+-d)\ \text{times}}\big),$$
$$ \overline D_j :=  -\diag \big(\underbrace{p(\lambda_j),\dots,p(\lambda_j)}_{(a_- -d)\ \text{times}},\underbrace{\lambda_j+\alpha,\dots,\lambda_j+\alpha}_{(a_0+d)\ \text{times}},\underbrace{\lambda_j-\alpha+1,\dots,\lambda_j-\alpha+1}_{d\ \text{times}}, \underbrace{1,\dots,1}_{(a_+-d)\ \text{times}}\big).$$

To prove
$$\mathcal{W}\left(M\big(\alpha, \boldsymbol a, K(h)\big)\right)\simeq\mathcal F_-(\alpha)^{\oplus  (a_- - d)} \oplus \mathcal F_0^+(\alpha)^{\oplus (a_0+d)} \oplus \mathcal F_0^-(\alpha)^{\oplus d} \oplus \mathcal F_+(\alpha)^{\oplus  (a_+ -d)},$$
it suffices to show that there exists a family 
$$\big\{A(\lambda_j) \in \GL_n(\C): j \in \Z \big\}$$
such that for every $j \in \Z$,
$$A(\lambda_j) P_{({\boldsymbol{a}},\alpha)}(\lambda_j) K(\lambda_j) =D_j A(\lambda_{j+1}),$$
$$ A(\lambda_{j+1}) K^{-1}(\lambda_j) \overline{P}_{({\boldsymbol{a}},\alpha)}(\lambda_j)  
= \overline D_j  A(\lambda_j).$$

Since
$$\overline{P}_{({\boldsymbol{a}},\alpha)}(\lambda_j) P_{({\boldsymbol{a}},\alpha)}(\lambda_j) =  \overline D_j D_j = -p(\lambda_j) \I_n,\quad \text{for all} \;\;\; j  \in \Z \setminus \{0,N\}, $$
 the above two equations are equivalent for all $j  \in \Z \setminus \{0,N\}$. From Remark \ref{d-intersectiondim}, $d = \dim\big(\mathcal K_{\alpha, \boldsymbol a, K} \mathcal S^- \cap \mathcal S^+ \big)$. Hence 
 $$d = \dim\big( \mathcal S^- \cap \mathcal K_{\alpha, \boldsymbol a, K}^{-1}\mathcal S^+ \big).$$
 
 Fix a basis $\{w_1,\dots,w_d\}$ of $\mathcal S^- \cap \mathcal K_{\alpha, \boldsymbol a, K}^{-1}\mathcal S^+$ and extend it to a basis of $\mathcal S^-$ and $\mathcal K_{\alpha, \boldsymbol a, K}^{-1}\mathcal S^+$ respectively as follows
 $$u_1, \dots, u_{a_--d}, w_1,\dots,w_d;\qquad w_1,\dots,w_d, v_1,\dots, v_{a_+ -d}.$$
 
 Then $\{u_1, \dots, u_{a_--d}, w_1,\dots,w_d, v_1,\dots, v_{a_+ -d} \}$ is a linearly independent set. Extend this to $\C^n$ by adding $z_1, \dots, z_{a_0 +d}$. We now define $A(\lambda_N)$ as follows
$$ A(\lambda_N) u_i = \mathsf e_i\;\; (1\leq i \leq a_--d), \quad A(\lambda_N) z_i = \mathsf e_{a_--d +i}\;\; (1\leq i \leq a_0+d),$$
$$ A(\lambda_N) w_i = \mathsf e_{a_- +a_0+i}\;\; (1\leq i \leq d), \quad A(\lambda_N) v_i = \mathsf e_{a_- +a_0+d +i}\;\; (1\leq i \leq a_+-d).$$
  
  Then $A(\lambda_N) \in \GL_n(\C)$ and
  $$A(\lambda_N) \mathcal S^- = \mathcal S^-_{a_- -d} \oplus \mathcal S^+_d, \qquad  A(\lambda_N) \mathcal K_{\alpha, \boldsymbol a, K}^{-1}\mathcal S^+ =  \mathcal S^+. $$ 

Since 
$$
A(\lambda_N)\mathcal S^-=\mathcal S^-_{a_--d}\oplus \mathcal S^+_d,
$$
standard linear algebra reasoning gives $A(\lambda_{N+1})\in \GL_n(\C)$ such that
$$
A(\lambda_N)P_{(\boldsymbol a,\alpha)}(\lambda_N)K(\lambda_N)
=
D_NA(\lambda_{N+1}),
$$
$$
A(\lambda_{N+1})K^{-1}(\lambda_N)\overline P_{(\boldsymbol a,\alpha)}(\lambda_N)
=
\overline D_NA(\lambda_N).
$$

Indeed, the displayed condition on $A(\lambda_N)\mathcal S^-$ is precisely the compatibility condition for solving these two equations at the singular value $\lambda_N$, and the resulting partially defined map extends to an element of $\GL_n(\C)$.
 
Note that
 $$A(\lambda_1) =\left(\prod_{j=1}^{N-1}D_j\right)\, A(\lambda_N)\, \left(\prod_{j=1}^{N-1}\left(P_{({\boldsymbol{a}},\alpha)}(\lambda_j) K(\lambda_j)\right)\right)^{-1},$$
where the empty product (i.e. $N =1$) is $\I_n$. Since $ \prod_{j=1}^{N-1}D_j$ preserves $ \mathcal S^+$, it follows that
  $$A(\lambda_1)K(\lambda_0)^{-1} \mathcal S^+ =  \mathcal S^+.$$

Since
$$
A(\lambda_1)K(\lambda_0)^{-1}\mathcal S^+=\mathcal S^+,
$$
standard linear algebra reasoning gives $A(\lambda_0)\in\GL_n(\C)$ such that
$$
A(\lambda_0)P_{(\boldsymbol a,\alpha)}(\lambda_0)K(\lambda_0)
=
D_0A(\lambda_1),
$$
$$
A(\lambda_1)K^{-1}(\lambda_0)\overline P_{(\boldsymbol a,\alpha)}(\lambda_0)
=
\overline D_0A(\lambda_0).
$$

Indeed, the displayed condition is exactly the compatibility condition for solving the two singular equations at $\lambda_0$, and the remaining values of $A(\lambda_0)$ may be chosen by extending to a basis. 

For $j \leq -1$ or $j \geq N+2$, the matrices $A(\lambda_j)$ can be determined from $A(\lambda_0)$ and $A(\lambda_{N+1})$. This completes the proof. 
\end{proof}

\begin{example}
As an illustration of the last theorem, we consider the module
$$
\mathcal M:=\mathcal{W}\left(M\big(1,(1,0,1),\T_2\big)\right)
= \Span\big\{v_{1,\lambda}, v_{2,\lambda}: \lambda\in \C \big\},
$$
and let $\mathcal M[\overline{\mu}]$ denote the $\mathfrak{sl}(2)$-submodule of
$\mathcal M$ with support $\overline{\mu}=\mu+\mathbb Z$.

For each coset $\overline{\mu}\in\mathbb C/\mathbb Z$, fix a representative
$\mu\in\mathbb C$, and define
$$\mathcal M^+[\overline{\mu}]
:=
\bigoplus_{k \in \Z}
\left(\C v_{1,\mu +2k+1} \oplus \C v_{2,\mu +2k}\right), \quad
\mathcal M^-[\overline{\mu}]
:=
\bigoplus_{k \in \Z}
\left(\C v_{1,\mu +2k} \oplus \C v_{2,\mu +2k+1}\right).
$$

Then both $\mathcal M^+[\overline{\mu}]$ and $\mathcal M^-[\overline{\mu}] $ are $\mathfrak{sl}(2)$-submodules of $\mathcal M [\overline{\mu}]$, and 
$$\mathcal M[\overline{\mu}]  = \mathcal M^+[\overline{\mu}] \oplus \mathcal M^-[\overline{\mu}].$$

Therefore, 
$$\mathcal M = \mathcal M^+ \oplus \mathcal M^-, \quad \text{where}\quad \mathcal M^{\pm}:= \bigoplus_{\overline{\mu} \in \C/\Z}  \mathcal M^{\pm}[\overline{\mu}].$$

A direct check shows that $\mathcal M^+ \simeq \mathcal F_0^+(1)$ and $\mathcal M^- \simeq \mathcal F_0^-(1)$, leading to 
$$\mathcal M \simeq \mathcal F_0^-(1) \oplus \mathcal F_0^+(1).$$

\end{example}

\section{Finite-dimensional submodules and quotients}
 In this section, we study modules in $\mathcal M (n)$ and their weightings, focusing on the existence of finite-dimensional submodules and quotients.
 \subsection{Finite-dimensional quotients of objects in $\mathcal M(n)$}
 \begin{proposition}\label{quotient-weighting}
Let $M\in\mathcal M$, and let $Q$ be a finite-dimensional $\mathfrak{sl}(2)$-module. Then 
\begin{itemize}
\item[(i)]$\Hom_{\mathfrak{sl}(2)}(\mathcal W(M),Q) \simeq\Hom_{\mathfrak{sl}(2)}(M,Q)$;
\item[(ii)] $Q$ is a quotient of $M$ if and only if $Q$ is a quotient of $\mathcal W(M)$.
\end{itemize}
\end{proposition}

\begin{proof}
Since $Q$ is finite dimensional, it is a weight module. Write
$$
Q=\bigoplus_{\lambda\in\Supp(Q)} Q_\lambda,
$$
where  $\lambda\in\Supp(Q)$ if $Q_{\lambda} \neq 0$. Let $\pr_\lambda:Q\to Q_\lambda$ denote the corresponding projection.

For $\psi\in\Hom_{\mathfrak{sl}(2)}(M,Q)$, define
$$
\widetilde{\psi}\bigl(m+(h-\lambda)M\bigr)
:=
\pr_\lambda(\psi(m)).
$$

This is well defined, and a direct check using
$eQ_\lambda\subseteq Q_{\lambda+1}$ and $fQ_\lambda\subseteq Q_{\lambda-1}$ shows that
$\widetilde{\psi}\in\Hom_{\mathfrak{sl}(2)}(\mathcal W(M),Q)$. Conversely, for $\varphi\in\Hom_{\mathfrak{sl}(2)}(\mathcal W(M),Q)$, define
$$
\widehat{\varphi}(m)
:=
\sum_{\lambda\in\Supp(Q)}
\varphi\bigl(m+(h-\lambda)M\bigr).
$$

The sum is finite, and the same weight-shift identities show that
$\widehat{\varphi}\in\Hom_{\mathfrak{sl}(2)}(M,Q)$. The two constructions are inverse to each other, and part (i) follows.

It remains only to show that this Hom-space isomorphism preserves surjectivity. More precisely, if
$\psi:M\twoheadrightarrow Q$, then $\widetilde{\psi}:\mathcal W(M)\twoheadrightarrow Q$; and if
$\varphi:\mathcal W(M)\twoheadrightarrow Q$, then $\widehat{\varphi}:M\twoheadrightarrow Q$.
This is immediate in the first  direction. In the other direction, if $\varphi:\mathcal W(M)\twoheadrightarrow Q$ and $v\in Q_{\lambda_i}$, choose
$m\in M$ such that
$$
\varphi\bigl(m+(h-\lambda_i)M\bigr)=v.
$$

If $\Supp(Q)=\{\lambda_1,\dots,\lambda_r\}$, then
$$
\widehat{\varphi} \left( \prod_{j\neq i}\frac{h-\lambda_j}{\lambda_i-\lambda_j}\,m\right) = v.
$$

Hence $\widehat{\varphi}$ is surjective, which completes the proof. \end{proof}
 \begin{proposition} \label{generalcaseprop}
Let $\alpha \in\C$, $\boldsymbol a=(a_-,a_0,a_+)\in \Z_{\geq 0}^3$ so that $a_- + a_0 + a_+ =n \geq 1$, and let $K(h)\in\GL_n(\C[h])$. Then  $M\big( \alpha, {{\boldsymbol a}}, K(h)\big)$ contains a proper submodule $N$ of rank $n$ if and only if  either 
$$a_0+d>0\quad \text{and} \quad \alpha \in 1+ \tfrac{1}{2}\Z_{\geq 0},$$ 
or 
$$d>0\quad \text{and} \quad \alpha \in \tfrac{1}{2}\Z_{\leq 0},$$ 
where $d:= a_-+ a_+ -\rank\left(\boldsymbol M_{\boldsymbol a}\big(\mathcal K_{\alpha, \boldsymbol a, K} \big)\right)$.
\end{proposition}

\begin{proof} 
We note that $M\big( \alpha, {{\boldsymbol a}}, K\big)$ has a proper submodule $N$ of rank $n$ precisely when the quotient $M\big( \alpha, {{\boldsymbol a}}, K\big) / N $ is a finite-dimensional $\mathfrak{sl}(2)$-module. By Theorem \ref{cohfamilydecomp}, the coherent family $\mathcal{W}\left(M\big(\alpha, \boldsymbol a, K(h)\big)\right)$ has a finite-dimensional quotient if and only if either 
$$ \left[\mathcal{W}\left(M\big(\alpha, \boldsymbol a, K(h)\big)\right): \mathcal F_0^-(\alpha)\right] \neq 0\quad \text{and}\quad \alpha \in \tfrac{1}{2}\Z_{\leq 0},$$
or
 $$\left[\mathcal{W}\left(M\big(\alpha, \boldsymbol a, K(h)\big)\right): \mathcal F_0^+(\alpha)\right] \neq 0\quad \text{and}\quad \alpha \in 1+ \tfrac{1}{2}\Z_{\geq 0}.$$

Since 
$$\left[\mathcal{W}\left(M\big(\alpha, \boldsymbol a, K(h)\big)\right): \mathcal F_0^-(\alpha)\right] = d, \quad \left[\mathcal{W}\left(M\big(\alpha, \boldsymbol a, K(h)\big)\right): \mathcal F_0^+(\alpha)\right] = a_0+d,$$
then Proposition \ref{quotient-weighting} implies the statement.
\end{proof}

\begin{proposition} \label{SESgeneralcase}
	Let $k \geq 0$, $K(h) \in \GL_n(\C[h])$, and suppose that $a_0 \geq 1$. Then the following non-split short exact sequences hold.
	$$0 \longrightarrow  M\big( 1+\tfrac{k}{2}, (a_-,-a_0,0), K(h+k+1) \big)  \longrightarrow  M\big( 1+\tfrac{k}{2}, (a_-,a_0,0) , K(h) \big)  \longrightarrow  L(k)^{\oplus a_0} \longrightarrow 0,$$
	$$0 \longrightarrow  M\big( 1+\tfrac{k}{2}, (0,-a_0,a_+), K(h-k-1) \big)  \longrightarrow  M\big( 1+\tfrac{k}{2}, (0,a_0,a_+), K(h) \big)  \longrightarrow  L(k)^{\oplus a_0} \longrightarrow 0,$$
	where $  L(k)$ is the simple $\mathfrak{sl}(2)$-module of dimension $k+1$.
\end{proposition}

\begin{proof} 
For the first short exact sequence, define 
$$ A(h):= \prod_{i=0}^k \left(P_{\big((a_-,a_0,0), 1+ \tfrac{k}{2} -i \big)}(h) K(h+i)\right) .$$

Since 
$$ P_{\big((a_-,a_0,0), 1+ \tfrac{k}{2} -i \big)}(h+1) = P_{\big((a_-,a_0,0), 1+ \tfrac{k}{2} -(i+1) \big)}(h),$$
and
$$ P_{\big((a_-,a_0,0), 1+ \tfrac{k}{2} -(k+1) \big)}(h) = P_{\big((a_-,-a_0,0), 1+ \tfrac{k}{2}\big)}(h),$$
it follows that
$$ P_{\big((a_-,a_0,0), 1+ \tfrac{k}{2}\big)}(h)  K(h)  A(h+1) = A(h) P_{\big((a_-,-a_0,0), 1+ \tfrac{k}{2}\big)}(h) K(h+k+1).$$

Then  $N_1:=A(h)\C[h]^{\oplus n}$ is a proper submodule of $M\big( 1+\tfrac{k}{2}, (a_-,a_0,0), K(h) \big)$, and 
$$ N_1 \simeq M\big( 1+\tfrac{k}{2}, (a_-,-a_0,0), K(h+k+1) \big).$$ 

We next consider the quotient $Q:= M\big( 1+\tfrac{k}{2}, (a_-,a_0,0), K(h)\big) / N_1$. Since $Q$ has central character $(k+1)^2$ and
$$ \dim(Q) = \dim \big(\C[h]^{\oplus n}/ A(h)\C[h]^{\oplus n}\big)= \deg(\det(A(h)) = a_0(k+1).$$

Therefore, $Q \simeq L(k)^{\oplus a_0}$ and the first exact sequence follows. Similarly, define
$$ B(h):= \sigma\left(\prod_{i=1}^{k+1}\Big(-K(h-i)^{-1} \overline P_{\bigl((0,a_0,a_+),\,1+ \tfrac{k}{2} -i \bigr)}(h)\Big)\right).$$

Then  $N_2:=B(h)\C[h]^{\oplus n}$ is a proper submodule of $M\big( 1+\tfrac{k}{2}, (0,a_0,a_+), K(h) \big)$, and 
$$ N_2 \simeq M\big( 1+\tfrac{k}{2}, (0,-a_0,a_+), K(h-k-1) \big).$$

Moreover, 
 $$ M\big( 1+\tfrac{k}{2}, (0,a_0,a_+), K(h) \big) / N_2 \simeq L(k)^{\oplus a_0}.$$

 The sequences are non-split because the middle terms of both sequences are free over
$\C[h]$, hence contain no nonzero finite-dimensional $\mathfrak{sl}(2)$-submodules. This completes the proof of the proposition.
\end{proof}

\begin{remark}
The one-sided conditions $a_-=0$ or $a_+=0$ cannot be removed in the following sense. When both $a_-$ and $a_+$ are nonzero, rank $n$ submodules may
still exist, but the corresponding $K$-matrix need not be obtained by a
simple shift of $K(h)$ as the following exact sequence for $\alpha =1$:
$$
0 \longrightarrow  M\left( 1, (1,-1,1), \begin{pmatrix}
1 & h+1 & 1 \\
0 & 1 & -1 \\
0 & 0& 1 
\end{pmatrix} \right)  \longrightarrow  M\big( 1,(1,1,1)  , \J_3 \big)  \longrightarrow  L(0) \longrightarrow 0,
$$
where the submodule  of $M\big( 1,(1,1,1)  , \J_3 \big)$  is $A(h)\C[h]^{\oplus 3}$ with $
A(h)=
\begin{pmatrix}
1&0&0\\
0&h&1\\
0&0&-1
\end{pmatrix}$.
\end{remark}

 \subsection{Finite-dimensional submodules of weightings of objects in $\mathcal M(n)$}
\begin{proposition} \label{fin-sub-coherentfam}
Let $n \geq 1$ and $M \in \mathcal M(n)$. Then $\mathcal W(M)$ has a finite-dimensional submodule $S$ if and only if there exists a short exact sequence
$$
0\longrightarrow M\longrightarrow X\longrightarrow S\longrightarrow 0,
$$
where $X \in \mathcal M(n)$.
\end{proposition}

\begin{proof}
We first relate finite-dimensional submodules of $\mathcal W(M)$ to finite-dimensional quotients of same-rank extensions of $M$. Suppose that
$$
0\longrightarrow M\longrightarrow X\longrightarrow S\longrightarrow 0
$$
is exact, where $X\in\mathcal M(n)$ and $S\neq0$ is finite dimensional. As usual, for $\lambda \in \C$, let $S_\lambda$ be the corresponding weight space. If $v\in S_\lambda$, choose a lift $x\in X$. Since $(h-\lambda)v=0$, we have $(h-\lambda)x\in M$. Define
$$\delta_\lambda:  S_\lambda \to M/(h-\lambda)M, \quad \text{where} \quad v\mapsto (h-\lambda)x+(h-\lambda)M.$$

Note that $\delta_\lambda$ is independent of the choice of $x$. Furthermore, $\delta_\lambda$ is injective. Indeed, suppose that $(h-\lambda)x\in (h-\lambda)M$. Then there exists $m\in M$ such that $(h-\lambda)x=(h-\lambda)m$. Since $X$ is $\C[h]$-free, we have  $x=m\in M$, and thus $v=\pi (x) = 0$. Define
$$
\delta:S\to \mathcal W(M),
\quad \text{where}\quad
\delta|_{S_\lambda}:= \delta_\lambda,
$$
for all $\lambda \in \Supp(S)$. We claim that $\delta$ is a homomorphism of  $\mathfrak{sl}(2)$-modules. Let $v\in S_\lambda$, and choose a lift $x\in X$. Then
$$
\delta(ev)=(h-\lambda-1)ex+(h-\lambda-1)M .
$$

On the other hand, in $\mathcal W(M)$, we have $(h-\lambda-1)e=e(h-\lambda)$, which gives
$$
e\delta(v)=e\big((h-\lambda)x+(h-\lambda)M\big)=(h-\lambda-1)ex+(h-\lambda-1)M.
$$

Thus $\delta(ev)=e\delta(v)$. Similarly, $\delta(fv)=f\delta(v)$. Therefore $\delta$ is an $\mathfrak{sl}(2)$-homomorphism, and its image is a finite-dimensional submodule of $\mathcal W(M)$.

\noindent\smallskip
Conversely, let $S\subset\mathcal W(M)$ be a nonzero finite-dimensional submodule. Write
$$
S=\bigoplus_{\lambda\in  \Supp(S)} S_\lambda,\qquad S_\lambda\subset M/(h-\lambda)M,
$$
where $|\Supp(S)| < \infty$. Let $\pi_\lambda:M\to M/(h-\lambda)M$ be the natural projection, and 
$\widetilde S_\lambda:=\pi_\lambda^{-1}(S_\lambda)$. Inside $\mathbb C(h)\otimes_{\C[h]} M$, define
$$
X:=\sum_{\lambda \in \Supp(S)} (h-\lambda)^{-1}\widetilde S_\lambda,
$$
where $(h-\lambda)^{-1}\widetilde S_\lambda:=\left\{(h-\lambda)^{-1}m\mid m\in\widetilde S_\lambda\right\}$. Then $X$ is a finitely generated torsion-free $\C[h]$-module of rank $n$, and
hence is $\C[h]$-free of rank $n$. Moreover, $X$ is stable under the actions of $e$ and $f$. Indeed, if $m\in\widetilde S_\lambda$, then
$$
e\left((h-\lambda)^{-1}m\right)=(h-\lambda-1)^{-1}em\in X , 
$$
and similarly for $f$. Therefore $X$ is an $\mathfrak{sl}(2)$-submodule of $\C(h)\otimes_{\C[h]} M$.
Next, define
$$
\Psi:X\longrightarrow S
$$
as follows. For $m\in M$, set $\Psi(m)=0$, and for $m\in\widetilde S_\lambda$, set
$$
\Psi\left((h-\lambda)^{-1}m\right)
=
m+(h-\lambda)M\in S_\lambda .
$$

This is a well defined map whose kernel is exactly $M$. Hence $X/M\simeq S$. Thus $X\in\mathcal M(n)$, and
$X/M\simeq S$ is finite dimensional. 
\end{proof}


\section{Socle filtration of  $M\big(\alpha, {\boldsymbol a}, \J_n\big)$}
Recall that   $\C[h]^m=\C[h]^{\oplus m}$. For $P_{(\boldsymbol a,\alpha)}(h)=\diag(p_1(h),\dots,p_n(h))$, where each $p_i(h)$ is a monic divisor of $\lambda_{\alpha} (h)$, we often use the notation
$$
M\big((p_1,\dots,p_n),\alpha,K(h)\big) 
$$
for the module $M(\alpha,\boldsymbol a,K(h))$. For $k \in \Z$, define $k^{+}: = \max \{0,k\}$. 

\subsection{The case of an upper-triangular matrix} \label{generalupper}
In this subsection, for $n \geq 1$, let
$$
A(h):=\begin{pmatrix}
a_{11} & u_{12}(h) & \cdots & u_{1n}(h) \\
0      & a_{22}    & \ddots & \vdots \\
\vdots & \ddots    & \ddots & u_{n-1,n}(h) \\
0      & \cdots    & 0      & a_{nn}
\end{pmatrix},
$$
where $a_{ii} \in \C^*$ for $i \in \{1,\dots, n\}$, and  $u_{jk}(h) \in \C[h]$ for all $1\leq j <k \leq n$. We recall Lemma 3.1 from \cite{GNZ}.

\begin{lemma}\label{surj-lem-ab}
Let $a,b\in\C^*$. Define
$$
T_{a,b}:\C[h]\longrightarrow\C[h],\qquad
T_{a,b}\bigl(g(h)\bigr)=a\,g(h+1)-b\,g(h).
$$
Then $T_{a,b}$ is a surjective linear map.
\end{lemma}

\begin{lemma}
The matrix $A(h)$ is $\sigma^{-1}$-similar to $\diag(a_{11},a_{22},\dots,a_{nn})$.
\end{lemma}
 
 \begin{proof}
For $1\leq k \leq n-1$, let 
$$U^{(k)}(h):= \I_n + \sum_{i=1}^{n-k}p_{i,i+k}(h) E_{i,i+k}, \quad \text{where} \quad p_{i,i+k}(h) \in \C[h].$$

A direct computation shows that, for each $1\leq i \leq n-1$, the $(i,i+1)$-entry of the matrix $\big(U^{(1)}(h)\big)^{-1} A(h)U^{(1)}(h+1)$ is given by 
$$u_{i,i+1}(h) + a_{ii}\ p_{i,i+1}(h+1) - a_{i+1,i+1}\ p_{i,i+1}(h).$$ 

By Lemma \ref{surj-lem-ab}, there exists $p_{i,i+1}(h)$, for each $1\leq i \leq n-1$ such that 
$$\left(\big(U^{(1)}(h)\big)^{-1} A(h)U^{(1)}(h+1)\right)_{i,i+1} = 0. $$

Applying the same argument successively to the transformed matrices, and observing that $\sigma^{-1}$-conjugation by $U^{(k)}(h)$ leaves the first $k-1$ superdiagonals unchanged, we obtain $U^{(2)}(h),\dots,U^{(n-1)}(h)$ such that
$$\Big(\prod_{i=1}^{n-1}U^{(i)}(h) \Big)^{-1} A(h) \prod_{i=1}^{n-1}U^{(i)}(h+1) = \diag(a_{11},a_{22},\dots,a_{nn}).$$ 
 
 This proves the lemma.
 \end{proof}
 
 \begin{corollary} \label{cor-scalar}
 Let $\alpha \in \C$ and $p \in \big\{1, h-\alpha+1, -\lambda_\alpha(h)\big\}$. Then
 $$M\big((p,\dots,p),\alpha,A(h)\big) \simeq \bigoplus_{i=1}^{n} M(p,\alpha, a_{ii}).$$ 
 \end{corollary}
 
 For the matrix $A(h):= \big(A(h)_{ij}\big)_{1\leq i,j \leq n}$  and for each $1\leq k \leq n$, we define the upper-left and lower-right $k \times k$ submatrix of $A(h)$, respectively, by
 $$A(h)^{[k]}:= \big(A(h)_{ij}\big)_{1\leq i,j \leq k}, \qquad A(h)_{[k]}:= \big(A(h)_{ij}\big)_{n-k+1\leq i,j \leq n}.$$
 
 \begin{lemma}\label{lem-cut-extensionnew}
Let $n \geq 1$, $1\leq m \leq n-1$, and suppose that 
$$P_{(\boldsymbol a,\alpha)}(h)=\diag(p_1(h),\dots,p_n(h)).$$ Then there exists a short exact sequence of $\mathfrak{sl}(2)$-modules:
\begin{multline*}
0\longrightarrow M\big((p_1,\dots,p_m), \alpha, A(h)^{[m]}\big)\longrightarrow
M\big((p_1,\dots,p_n), \alpha, A(h)\big)\\
\longrightarrow M\big((p_{m+1},\dots,p_n), \alpha, A(h)_{[n-m]}\big)\longrightarrow 0.
\end{multline*}
\end{lemma}
 
\subsection{The case of a Jordan cell}
Throughout this subsection, we assume that $A(h)=\J_n$. For $k\leq n$, we view $M\big((p_1,\dots,p_k),\alpha,\J_k\big)$ as the $\C[h]$-submodule of $M\big((p_1,\dots,p_n),\alpha,\J_n\big)$  spanned by the first $k$ standard basis vectors (equivalently, $\C[h]^k\oplus \{0\}\subset \C[h]^k\oplus\C[h]^{n-k}$). One checks easily that it is also an $\mathfrak{sl}(2)$-submodule.

 \begin{lemma}\label{lem:diag-Jn-explicit}
Define $P_n(h)=\big(P_{i,j}(h)\big)_{1 \leq i,j \leq n} \in\Mat_n(\C[h])$ by
$$
P_{i,j}(h)=
\begin{cases}
(-1)^{\,j-i}\displaystyle\binom{h+j-i-1}{j-i},& i\le j,\\
0,& i>j.
\end{cases}
$$
Then $P_n(h)\in \GL_n(\C[h])$ and $P_n(h)^{-1}\J_nP_n(h+1)=\I_n$.
\end{lemma}

\begin{proof} The identity follows by a direct computation using that for $r\geq 1$ we have 
$$
p_r(h+1)+p_{r-1}(h+1)=p_r(h),
$$
where $p_r(h)=(-1)^r\binom{h+r-1}{r}$.
\end{proof}
\begin{remark} We notice that $P_n(h)^{-1} = P_n(-h)$. In particular, 
$$
P_n(h)^{-1}:=
\begin{pmatrix}
1 & \binom{h}{1} & \binom{h}{2} & \cdots & \binom{h}{n-1} \\
0 & 1 & \binom{h}{1} & \cdots & \binom{h}{n-2} \\
0 & 0 & 1 & \ddots & \vdots \\
\vdots & \vdots & \ddots & \ddots & \binom{h}{1} \\
0 & 0 & \cdots & 0 & 1
\end{pmatrix}.
$$
\end{remark} 
\begin{proposition}\label{prop-ppp-directsum}
Let $\alpha \in \C$ and $p \in \big\{1, h-\alpha+1, -\lambda_\alpha(h)\big\}$. For $1\leq j\leq n$, define
$$
u_j:=P_n(h)\mathsf e_j
=\sum_{r=0}^{j-1}(-1)^r\binom{h+r-1}{r}\,\mathsf e_{j-r}.
$$
Then
$$
M\big((p,\ldots,p), \alpha,\J_n\big)=\C[h]u_1\oplus\C[h]u_2\oplus\cdots\oplus\C[h]u_n
$$
and each $\C[h]u_j$ is an $\mathfrak{sl}(2)$-submodule isomorphic to the rank-one module $M(p,\alpha,1)$.
\end{proposition}

\begin{proof}
By Lemma~\ref{lem:diag-Jn-explicit}, $P_n(h)^{-1}\J_nP_n(h+1)=\I_n$, hence
$$
F':=P_n(h)^{-1}FP_n(h+1)=P_n(h)^{-1}\,p(h)\J_n\,P_n(h+1)=p(h)\I_n.
$$

Let $E':=P_n(h)^{-1}EP_n(h-1)$. Then, by  Proposition~1.3  we have an
$\mathfrak{sl}(2)$--module isomorphism
$$
M(E,F)\xrightarrow{\simeq} M(E',F'),\qquad v\mapsto P_n(h)^{-1}v.
$$

In $M(E',F')$ we have $F'\sigma^{-1}(E')=\lambda_\alpha(h)\I_n$, which implies $E'=\sigma\Big(\frac{\lambda_\alpha(h)}{p(h)}\Big)\,\I_n$.
Therefore,  both $e$ and $f$ act {diagonally} on $\C[h]\mathsf e_j$, and more precisely, 
$$
f\cdot (g \mathsf e_j)=p(h)\,\sigma^{-1}(g)\,\mathsf e_j,\qquad
e\cdot (g \mathsf e_j)=\sigma\Big(\frac{\lambda_\alpha(h)}{p(h)}\Big)\sigma(g)\,\mathsf e_j.
$$

Therefore $M(E',F')=\bigoplus_{j=1}^n \C[h]\mathsf e_j$ as $\mathfrak{sl}(2)$--modules, and each $\C[h]\mathsf e_j\simeq M(p, \alpha, 1)$, which completes the proof.
\end{proof}

\begin{lemma}\label{lem-cut-extension} 
Let $1\leq m\leq n-1$ and let $p_1,\dots,p_n\in\C[h]$ satisfy $\deg p_m <\deg p_{m+1}$.
Then the following short exact sequence is non-split
$$
0\rightarrow M\big((p_1,\dots,p_m), \alpha,\J_{m}\big)\rightarrow
M\big((p_1,\dots,p_n), \alpha,\J_n\big)\rightarrow
M\big((p_{m+1},\dots,p_n), \alpha,\J_{n-m}\big)\rightarrow 0.
$$
\end{lemma}

\begin{proof} Let $k=n-m$. The $f$--matrix on $M\big((p_1,\dots,p_n), \alpha,\J_n\big)$ is 
$$F=\begin{pmatrix}\diag(p_1,\dots,p_m) \J_m&p_m E_{m,1}\\0&\diag(p_{m+1},\dots,p_n)\J_k\end{pmatrix}.$$

  Write $\C[h]^n=\C[h]^m\oplus\C[h]^k$ using the standard basis and set
$U:=\C[h]^m\oplus\{0\}$. Then, $U$ is an $\mathfrak{sl}(2)$-submodule of $M\big((p_1,\dots,p_n), \alpha,\J_n\big)$, and
$$U\simeq M\big((p_1,\dots,p_m),\alpha,\J_m\big),\quad M\big((p_1,\dots,p_n), \alpha,\J_n\big)/U\simeq M\big((p_{m+1},\dots,p_n), \alpha,\J_k\big).$$

Suppose, for contradiction, that the exact sequence splits. Then there exists an $\mathfrak{sl}(2)$-submodule $N$ for which $M\big((p_1,\dots,p_n), \alpha,\J_n\big)=U\oplus N$. Hence there exists $L(h)\in\Mat_{m\times k}(\C[h])$ so that
$$N=\big\{(L(h)u\;\;\;u)^{\mathsf T}: u\in\C[h]^k\big\}\subset \C[h]^m\oplus\C[h]^k.$$

The condition $f\cdot N\subset N$ is equivalent to the identity
\begin{equation} \label{eq-p-m-l}
\diag(p_1,\dots,p_m) \J_m L(h+1)+p_m E_{m,1}\;=\;L(h)\diag(p_{m+1},\dots,p_n)\J_k.
\end{equation}

Taking the $(m,1)$-entry of \eqref{eq-p-m-l} and setting $\ell(h):=L_{m,1}(h)$, we obtain
$$
p_m(h)\ell(h+1)+p_m(h)\;=\;p_{m+1}(h)\ell(h).
$$

If $\ell=0$ then $p_m=0$, which is impossible. The case $\ell=-1$ is easy. If $\ell\neq0$ and $\ell\neq -1$, then taking the degrees of both sides of the polynomial equation, we have a contradiction with $\deg p_m<\deg p_{m+1}$. Hence, no such $L(h)$ exists, so the exact sequence is non-split.
\end{proof}

\begin{lemma}\label{lem-diffeq-2}
Let $p_1(h),p_2(h)\in \C[h]$ such that $\deg p_i=i$, and $p_1\mid p_2$. Given $m,k,l\in \Z_{\geq 0}$ with $n=m+k+l\geq 1$, define
$$
F(h)=\diag(\underbrace{1,\ldots,1}_{m},\underbrace{p_1,\ldots,p_1}_{k},\underbrace{p_2,\ldots,p_2}_{l})\J_n.
$$
\begin{itemize}
\item[(i)] The equation
$$
F(h)\boldsymbol{q}(h+1)=p_1(h)\boldsymbol{q}(h)
$$
has a nonzero solution in $\C[h]^n$ if and only if $k\geq m+1$, and the solution space has dimension $k-m$.

\item[(ii)] The equation
$$
F(h)\boldsymbol{q}(h+1)=p_2(h)\boldsymbol{q}(h)
$$
has a nonzero solution in $\C[h]^n$ if and only if $l\geq 2m+k+1$, and the solution space has dimension $l-2m-k$.
\end{itemize}
\end{lemma}

\begin{proof}
For $r\in\Z$, let $V_r:=\{g\in\C[h] : \deg g<r\}$, with the convention $V_r=0$ if $r\le 0$. Then $\dim V_r=r^{+}$. For a nonzero polynomial $a(h)\in\C[h]$  define the linear operator
$$
T_a:\C[h]\to\C[h],\qquad T_a\big(g(h)\big):=a(h-1)g(h-1)-g(h).
$$

If  $\deg a>0$  and $g\neq 0$, then $\deg\big(T_a(g)\big)=\deg a+\deg g$. In particular, if $\deg a>0$ and  $r\geq 0$,
$$
T_a:V_r\hookrightarrow V_{r+\deg a}
$$
is injective. We also use the following claim, whose verification is straightforward.

\smallskip\noindent
{\bf Claim.}
Let $\ell \in \Z_{\geq 1}$ and $\boldsymbol x=(x_1 \;\; \dots \;\; x_\ell)^{\mathsf T}\in\C[h]^\ell$. Then $\J_\ell\boldsymbol x(h+1)=\boldsymbol x(h)$
if and only if
$$
x_{i+1} = T_1^{i}x_1 \;\;(1\leq i \leq \ell-1),\qquad \text{and} \qquad x_1 \in V_\ell.
$$

In particular, the space of solutions of $\J_\ell\boldsymbol x(h+1)=\boldsymbol x(h)$ has dimension $\ell$.

\smallskip\medskip

\noindent
We first prove part (i). Let $\boldsymbol q=(\boldsymbol u\;\;\boldsymbol v\;\;\boldsymbol w)^{\mathsf T}$, where $\boldsymbol u\in\C[h]^m$, $\boldsymbol v\in\C[h]^k$, and $\boldsymbol w\in\C[h]^l$. The equation $F(h)\boldsymbol q(h+1)=p_1(h)\boldsymbol q(h)$ becomes
\[
\begin{cases}
\J_m\boldsymbol u(h+1)+E_{m,1}\boldsymbol v(h+1)=p_1(h)\boldsymbol u(h),\\
\J_k\boldsymbol v(h+1)+E_{k,1}\boldsymbol w(h+1)=\boldsymbol v(h),\\
p_2(h)\J_l\boldsymbol w(h+1)=p_1(h)\boldsymbol w(h).
\end{cases}
\tag{1}
\]

The last equation of $p_2(h)\J_l\boldsymbol w(h+1)=p_1(h)\boldsymbol w(h)$ is
$$
p_2(h)w_l(h+1)=p_1(h)w_l(h),
$$
which implies $w_l=0$. Then  from the remaining equations of $p_2(h)\J_l\boldsymbol w(h+1)=p_1(h)\boldsymbol w(h)$, we obtain
$$
w_{l-1}=0,\quad w_{l-2}=0,\;\dots,\quad w_1=0,
$$
so $\boldsymbol w=0$. Then, the second equation in $(1)$ becomes $\J_k \boldsymbol v(h+1)=\boldsymbol v(h).$
By the above claim, its solution space is parametrized by $V_k$ via the first coordinate
$v_1$. Writing the first equation in $(1)$ coordinatewise and replacing $h$ by $h-1$, we obtain
$$
u_{i+1}=T_{p_1}(u_i)\qquad (1\le i\le m-1),
\qquad
v_1=T_{p_1}(u_m).
$$

Hence, $v_1=T_{p_1}^m(u_1)$. Therefore the polynomial solutions of $F(h)\boldsymbol q(h+1)=p_1(h)\boldsymbol q(h)$ are in bijection with the set of
$u_1\in\C[h]$ such that $T_{p_1}^m(u_1)\in V_k$. Since $\deg p_1=1$, the map 
$$
T_{p_1}^m:V_{k-m}\hookrightarrow V_k
$$
is injective, and $T_{p_1}^m(u_1)\in V_k$ holds exactly when $u_1\in V_{k-m}$. Thus the solution
space is in bijection with $V_{k-m}$, and hence has dimension $\dim V_{k-m}=(k-m)^{+}$.

\medskip\noindent
For part (ii), consider the equation $F(h)\boldsymbol q(h+1)=p_2(h)\boldsymbol q(h)$, which is equivalent to the system
\[
\begin{cases}
\J_m\boldsymbol u(h+1)+E_{m,1}\boldsymbol v(h+1)=p_2(h)\boldsymbol u(h),\\
p_1(h)\bigl(\J_k\boldsymbol v(h+1)+E_{k,1}\boldsymbol w(h+1)\bigr)=p_2(h)\boldsymbol v(h),\\
p_2(h)\J_l\boldsymbol w(h+1)=p_2(h)\boldsymbol w(h).
\end{cases}
\tag{2}
\]

Set $r(h):=\frac{p_2(h)}{p_1(h)}$. Then $r(h) \in \C[h]$ and $\deg r=1$. The last equation in $(2)$ simplifies to $\J_l\boldsymbol w(h+1)=\boldsymbol w(h)$, and hence, by the claim, its solution space is parametrized by $V_l$ via the first coordinate $w_1$. Dividing the middle equation in $(2)$ by $p_1(h)$, we obtain
$$
\J_k \boldsymbol v(h+1)+E_{k,1}\boldsymbol w(h+1)=r(h)\boldsymbol v(h).
$$

Writing this equation coordinatewise and replacing $h$ by $h-1$, leads to 
$$
v_{i+1}=T_r(v_i)\qquad (1\le i\le k-1),
\qquad
w_1=T_r(v_k).
$$

Hence $w_1=T_r^k(v_1)$. Since $\deg r=1$, the map 
$$
T_r^k:V_{l-k}\hookrightarrow V_l
$$
is injective, and the solutions $(\boldsymbol v\;\; \boldsymbol w)^{\mathsf T}$ of the lower subsystem are in bijection with $V_{l-k}$. Therefore, the solution space of the lower subsystem has dimension $\dim V_{l-k}=(l-k)^{+}.$

Finally, consider the first equation in $(2)$. Replacing $h$ by $h-1$, we obtain
$$
u_{i+1}=T_{p_2}(u_i)\qquad (1\leq i\leq m-1),
\qquad
v_1=T_{p_2}(u_m).
$$

Hence $v_1=T_{p_2}^m(u_1)$. Since $\deg p_2=2$, the map 
$$
T_{p_2}^m:V_{l-k-2m}\hookrightarrow V_{l-k}
$$
is injective. Combining this with the fact that $V_{l-k}$ parameterizes the solution spaces of the lower subsystem, we see that the solutions of $F(h)\boldsymbol q(h+1)=p_2(h)\boldsymbol q(h)$ are in bijection with $V_{l-k-2m}$. Hence, the solution space has dimension $\dim V_{l-k-2m}=(l-2m-k)^{+}$. This proves the lemma.
\end{proof}

\begin{lemma}\label{lem-splitting-C}
Let $\boldsymbol{a} = (a_-, a_0, a_+) \in \Z_{\geq 0}^3$ with $a_+ >2a_-+a_0$, and set $n=a_- + a_0 +a_+$.   Then
$$
M(\alpha, \boldsymbol{a},\J_n) \simeq M\big(\alpha,(a_-, a_0, a_+-1),\J_{n-1}\big)\oplus M(\alpha,(0,0,1),1).
$$
\end{lemma}

\begin{proof}
 Let $U$ be the submodule of $M(\alpha, \boldsymbol{a},\J_n)$ corresponding to the first $n-1$ coordinates. Then, 
 $$U \simeq M\big(\alpha,(a_-, a_0, a_+-1),\J_{n-1}\big).$$ 
 
 It remains to show that the short exact sequence
$$
0\longrightarrow U\longrightarrow M(\alpha, \boldsymbol{a},\J_n)\longrightarrow M\big(\alpha,(0,0,1),1\big)\longrightarrow 0.
$$
splits. Applying Lemma~\ref{lem-diffeq-2} to the $f$-matrix of $M(\alpha, \boldsymbol{a},\J_n)$, we obtain
$$
\dim \Hom_{\mathfrak{sl}(2)}\left(M\big(\alpha,(0,0,1),1\big),M(\alpha, \boldsymbol{a},\J_n)\right)  = a_+-2a_--a_0
$$
because the equation
$$
F(h) \boldsymbol q(h+1)=p_2(h)\boldsymbol q(h)
$$
has the solution space of dimension $a_+-2a_--a_0$. On the other hand, Lemma~\ref{lem-diffeq-2}, applied to the submodule $U\simeq M\big(\alpha,(a_-, a_0, a_+-1),\J_{n-1}\big)$, yields
$$
\dim \Hom_{\mathfrak{sl}(2)}\left(M\big(\alpha,(0,0,1),1\big),M\big(\alpha,(a_-, a_0, a_+-1),\J_{n-1}\big)\right)=a_+-2a_--a_0-1.
$$

Therefore, there exists a nonzero homomorphism $\varphi:M(\alpha,(0,0,1),1)\to M(\alpha, \boldsymbol{a},\J_n)$, such that $X:=\varphi(M(\alpha,(0,0,1),1)) \not\subset U$. Since $M(\alpha,(0,0,1),1)$ is simple and $\varphi\neq 0$, we have $X\simeq M(\alpha,(0,0,1),1)$. Because $X\not\subset U$, the composite map
$$
X\hookrightarrow M\twoheadrightarrow M(\alpha, \boldsymbol{a},\J_n)/U\simeq M(\alpha,(0,0,1),1)
$$
is nonzero. Since both $X$ and $M(\alpha,(0,0,1),1)$ are simple, this map is an isomorphism. In particular,
$$
X\cap U=0
\qquad\text{and}\qquad
M(\alpha, \boldsymbol{a},\J_n)=U+X,
$$
which proves the statement. 
\end{proof}

\begin{definition}
Let $M$ be an $\mathfrak{sl}(2)$--module. The \emph{socle} of $M$, denoted $\soc(M)$, is the sum of all simple submodules of $M$ (equivalently, the largest semisimple submodule of $M$).

Inductively define the \emph{socle filtration} $\{\soc_i(M)\}_{i\ge 0}$ by
$$
\soc_0(M):=0,\qquad 
\soc_{i+1}(M):=\pi_i^{-1}\big(\soc(M/\soc_i(M))\big)\qquad (i\ge 0),
$$
where $\pi_i:M\to M/\soc_i(M)$ is the natural projection.
The successive quotients
$$
\soc_i(M)/\soc_{i-1}(M)\qquad (i\ge 1)
$$
are called the \emph{socle layers} of $M$.

If $\soc_\ell(M)=M$ for some $\ell$, then the least such $\ell$ is called the \emph{socle length} of $M$.
\end{definition}

For the remainder of this section, we introduce the following notation. Given $a_0\in \Z_{\geq 1}$ and $a_+\in \Z_{\geq 0}$, let $N_{a_0,a_+}$ be the $\mathfrak{sl}(2)$-module isomorphic to $M\big(\alpha,(0,a_0,a_+),\J_{a_0+a_+}\big)$, whose $f$-matrix is
$$
F_{a_0,a_+}(h)=\begin{pmatrix} p(h)\I_{a_0} & p(h)c_{a_0}(h)\mathsf e_1^{\mathsf T}\\ 0 & -\lambda_\alpha(h)\J_{a_+}
\end{pmatrix},\qquad
c_{a_0}(h):=P_{a_0}(h)^{-1}\mathsf e_{a_0},\quad \mathsf e_1\in \C^{a_+},
$$
where $p(h):=h-\alpha+1$. Indeed, this follows from the identity  
\begin{equation} \label{changingbasis}
F_{a_0,a_+}(h)=\begin{pmatrix}P_{a_0}(h)&0\\0&\I_{a_+}\end{pmatrix}^{-1} P_{((0,a_0,a_+), \alpha)}(h) \J_{a_0+a_+}\,
\begin{pmatrix}P_{a_0}(h+1)&0\\0&\I_{a_+}\end{pmatrix}.
\end{equation}

By convention, $N_{0,a_+}\simeq M(\alpha,(0,0,a_+),\J_{a_+})$.

\begin{lemma}\label{lem-resonant-step}
Assume that $2\alpha-2\in\Z_{\geq 0}$. Define
$$
g_\alpha(h):=\prod_{j=0}^{2\alpha-2}(h+\alpha-1-j),
\qquad
D:=\C[h]\mathsf e_1\subset N_{a_0,a_+},
\qquad
S:=\C[h]g_\alpha(h)\mathsf e_1\subset D.
$$

Then 
$$D\simeq M(\alpha,(0,1,0),1),\quad S=\soc(D)\simeq M(\alpha,(0,-1,0),1),\quad N_{a_0,a_+}/S \simeq L(2\alpha-2)\oplus N_{a_0-1,a_+}. $$
\end{lemma}

\begin{proof}
By Lemma \ref{lem:diag-Jn-explicit}, 
$$
c_{a_0}(h)=P_{a_0}(h)^{-1}\mathsf e_{a_0}=
\left(\binom{h}{{a_0}-1},\;\; \binom{h}{{a_0}-2},\;\;\dots\;\;\binom{h}{1},\;\;1\right)^{\mathsf T}.
$$

The upper left $a_0\times a_0$ block of $F_{a_0,a_+}(h)$ is $p(h)\I_{a_0}$, so by
Proposition \ref{prop-ppp-directsum}, $D=\C[h]\mathsf e_1$
is an $\mathfrak{sl}(2)$--submodule isomorphic to $M(\alpha,(0,1,0),1)$.
Hence
$$
S=\soc(D)=\C[h]g_\alpha(h)\mathsf e_1\simeq M(\alpha,(0,-1,0),1),
\qquad
D/S\simeq L(2\alpha-2).
$$

Write
$$
c_{a_0}(h)=\binom{b_{a_0}(h)}{c_{{a_0}-1}(h)},
\qquad
b_{a_0}(h):=\binom{h}{{a_0}-1}.
$$

Then projection onto the last $({a_0}+{a_+}-1)$ coordinates yields the short exact sequence
$$
0\longrightarrow D/S\longrightarrow N_{a_0,a_+}/S\longrightarrow N_{a_0-1,a_+}\longrightarrow 0.
$$

Hence, it only remains to prove that this sequence splits. We look for an $\mathfrak{sl}(2)$-module complement $W$ to $D/S$ in $N_{a_0,a_+}/S$,
$$
W=\big\{(\eta(h)u\;\;\;u)^{\mathsf T}: u\in \C[h]^{{a_0}+{a_+}-1}\big\}, \qquad
\eta(h)\in \big(\C[h]/\C[h]g_\alpha(h)\big)^{1\times({a_0}+{a_+}-1)}.
$$

It is enough to take
$$
\eta(h)=\big(0\;\;\dots\;\;0\;\;y_1(h)\;\;\dots\;\;y_{a_+}(h)\big),
$$
supported only on the last $a_+$ coordinates. Then the condition
$f\cdot W\subset W$ is equivalent to

$$
p(h)y_1(h+1)+\lambda_\alpha(h)y_1(h)\equiv -p(h)b_{a_0}(h)\pmod{g_\alpha(h)},
$$
and
$$
p(h)y_j(h+1)+\lambda_\alpha(h)\bigl(y_{j-1}(h)+y_j(h)\bigr)\equiv 0
\pmod{g_\alpha(h)}
\qquad (2\leq j\leq {a_+}).
$$

After dividing by $p(h)$, these congruences become
$$
y_1(h+1)-(h+\alpha)y_1(h)\equiv -b_{a_0}(h)
\pmod{\prod_{j=0}^{2\alpha-3}(h+\alpha-1-j)},
$$
and
$$
y_j(h+1)-(h+\alpha)y_j(h)\equiv (h+\alpha)y_{j-1}(h)
\pmod{\prod_{j=0}^{2\alpha-3}(h+\alpha-1-j)}
\qquad (2\leq j\leq {a_+}).
$$

We now solve this system of congruences. If $\alpha=1$, then there is nothing to solve. Assume that $\alpha>1$, and set
$$
q_\alpha(h):=\prod_{j=0}^{2\alpha-3}(h+\alpha-1-j),\qquad d:=\deg q_\alpha=2\alpha-2.
$$

We claim that for every $r(h)\in\C[h]$ there exists $y(h)\in\C[h]$ such that
$$
y(h+1)-(h+\alpha)y(h)\equiv r(h)\pmod{q_\alpha(h)}.
$$

Indeed, for $0\leq i\leq d$, let $t_i:=-\alpha+1+i$. Then $t_0,\dots,t_{d-1}$ are precisely the roots of $q_\alpha(h)$.
Choose numbers $\beta_0,\dots,\beta_d$ recursively by
$$
\beta_0:=0,\qquad \beta_{i+1}:=(i+1)\beta_i+r(t_i)\qquad (0\leq i\leq d-1).
$$

By interpolation, there exists a polynomial $y(h)$ such that $y(t_i)=\beta_i$ for $0\leq i\leq d$. Then
$$
y(t_i+1)-(t_i+\alpha)y(t_i)=\beta_{i+1}-(i+1)\beta_i=r(t_i),\quad \text{for}\;\; 0\leq i\leq d-1.
$$

Hence
$$
y(h+1)-(h+\alpha)y(h)\equiv r(h)\pmod{q_\alpha(h)},
$$
proving the claim. 

Applying the claim first with $r(h)=-b_{a_0}(h)$, and then recursively with
$$r(h)=(h+\alpha)y_{j-1}(h),\qquad  (2\leq j\leq a_+),$$ 
we obtain the required polynomials $y_1,\dots,y_{a_+}$. Thus such submodule $W$ exists, and  
$$
N_{a_0,a_+}/S=(D/S)\oplus W \simeq L(2\alpha-2)\oplus N_{a_0-1,a_+}.
$$
\end{proof}

\begin{proposition}\label{prop-soc1-and-quotient}
Let $\boldsymbol{a} = (a_-, a_0, a_+) \in \Z_{\geq 0}^3$ and $n=a_- + a_0 +a_+$. Define
$$
A:=M(\alpha,(1,0,0),1),\qquad
B_{\pm}:=M(\alpha,(0,\pm1,0),1), ,\qquad
C:=M(\alpha,(0,0,1),1),
$$
$$
K:=\min\{a_-,a_0\},\qquad R:=\min\{a_+,2a_-+a_0\}.
$$
\begin{enumerate}
\item[(i)] If $2\alpha-2\notin\Z_{\ge 0}$, then
$$
\soc_1\big(M(\alpha, \boldsymbol{a},\J_n)\big)\simeq A^{\oplus {a_-}}\oplus
B_+^{\oplus{(a_0-a_-)}^{+}} \oplus
C^{\oplus{(a_+-2a_--a_0)}^{+}},
$$
and
$$
M(\alpha, \boldsymbol{a},\J_n)/\soc_1\big(M(\alpha, \boldsymbol{a},\J_n)\big) \simeq M(\alpha,(0,K,R),\J_{K+R}).
$$

\item[(ii)] If $2\alpha-2\in\Z_{\ge 0}$, then
$$
\soc_1\big(M(\alpha, \boldsymbol{a},\J_n)\big)\simeq A^{\oplus {a_-}}\oplus
B_-^{\oplus{(a_0-a_-)}^{+}} \oplus
C^{\oplus{(a_+-2a_--a_0)}^{+}},
$$
and
$$
M(\alpha, \boldsymbol{a},\J_n)/\soc_1\big(M(\alpha, \boldsymbol{a},\J_n)\big)\simeq
L^{\oplus{(a_0-a_-)}^{+}}\oplus M(\alpha,(0,K,R),\J_{K+R}),
$$
where $ L:=L(2\alpha-2)$.
\end{enumerate}
\end{proposition}

\begin{proof}
By repeated application of Lemma \ref{lem-splitting-C},
$$
M(\alpha, \boldsymbol{a},\J_n)\simeq M(\alpha,(a_-,a_0,R),\J_{a_-+a_0+R})\oplus C^{\oplus(a_+-R)}.
$$

Since $C$ is simple, it follows immediately that
$$C^{\oplus(a_+-R)}=C^{\oplus{(a_+-2a_--a_0)}^{+}} \subset \soc_1(M),$$ 
and the quotient by the socle is the same as for the first summand. Hence, it suffices to prove the proposition under the additional assumption $a_+\leq 2a_-+a_0$. (i.e. $a_+=R$).

From now on, we will work under this assumption. Define
$$
U(h):=\diag\big(P_{a_-}(h),P_{a_0}(h),\I_{a_+}\big) \in \GL_n (\C[h]),
$$
 and
$$ M(E,F) := M\left(U(h)^{-1} \sigma\left(\J_{n}^{-1}\overline{P}_{({\boldsymbol{a}},\alpha)}(h)\right) U(h-1),\; U^{-1}(h) P_{({\boldsymbol a}, \alpha)}(h) \J_{n} U(h+1)\right).$$

Hence, $M(E,F) \simeq M(\alpha, \boldsymbol{a},\J_n)$. By Lemma \ref{lem:diag-Jn-explicit}, 
$$
F(h)= U(h)^{-1} P_{({\boldsymbol a}, \alpha)}(h) \J_{n} U(h+1)=  
\begin{pmatrix}
\I_{a_-} & X(h) & 0\\
0 & p(h)\I_{a_0} & p(h)c_{a_0}(h)\mathsf e_1^{\mathsf T}\\
0 & 0 & -\lambda_\alpha(h)\J_{a_+}
\end{pmatrix},
$$
for some matrix $X(h) \in \Mat_{a_- \times a_0}(\C[h])$, $p(h):=h-\alpha+1$, and
$$
c_{a_0}(h):=P_{a_0}(h)^{-1}\mathsf e_{a_0}
=
\bigg(\binom{h}{{a_0}-1},\;\;\; \binom{h}{{a_0}-2},\;\;\;\dots\;\;\;\binom{h}{1},\;\;\;1
\bigg)^{\mathsf T}.
$$

In particular, if ${a_-}<{a_0}$, then the last ${a_-}$ entries of $c_{a_0}(h)$ are exactly $c_{a_-}(h)$.

Let
$$
\mathbf U:=\C[h]^{a_-}\oplus \{0\}\oplus \{0\}\subset M(E,F) = \C[h]^{a_-} \oplus \C[h]^{a_0} \oplus \C[h]^{a_+}. 
$$

Then $\mathbf U$ is a submodule, and it follows from Corollary \ref{cor-scalar} that $\mathbf U\simeq A^{\oplus a_-}$. Let $N:=M(E,F)/\mathbf U$, then $N \simeq N_{a_0,a_+}$. In particular, the induced $f$--matrix on $N$ is
$$
\overline F(h)=
\begin{pmatrix}
p(h)\I_{a_0} & p(h)c_{a_0}(h)\mathsf e_1^{\mathsf T}\\
0 & -\lambda_\alpha(h)\J_{a_+}
\end{pmatrix}.
$$

\smallskip\noindent For part (i), assume that $2\alpha-2\notin\Z_{\geq 0}$. By Corollary \ref{cor-scalar}, the multiplicity of $A$ in $\soc_1\big(M(E,F)\big)$, and hence in $\soc_1\big(M(\alpha,\boldsymbol a,\J_n)\big)$, is precisely $a_-$. Applying Lemma \ref{lem-diffeq-2} to the matrix  $P_{({\boldsymbol a}, \alpha)}(h) \J_{n}$, and using the assumption $a_+ \leq 2a_-+a_0$, we deduce that
$$
\dim\Hom_{\mathfrak{sl}(2)}\big(B_+,M(\alpha,\boldsymbol a,\J_n)\big)=(a_0-a_-)^{+},\qquad \dim\Hom_{\mathfrak{sl}(2)}\big(C,M(\alpha,\boldsymbol a,\J_n)\big)=0.
$$

Hence, $\soc_1\big(M(\alpha,\boldsymbol a,\J_n)\big)\simeq A^{\oplus a_-}\oplus B_+^{\oplus(a_0-a_-)^{+}}$. If $a_-\geq a_0$, then $(a_0-a_-)^{+}=0$, so $\soc_1\big(M(\alpha,\boldsymbol a,\J_n)\big)\simeq \mathbf U$ and $M(\alpha,\boldsymbol a,\J_n)/\soc_1\big(M(\alpha,\boldsymbol a,\J_n)\big)\simeq N$.
Since
$$
\overline F(h)=
\begin{pmatrix}
p(h)\I_{a_0} & p(h)c_{a_0}(h)\mathsf e_1^{\mathsf T}\\
0 & -\lambda_\alpha(h)\J_{a_+}
\end{pmatrix}.
$$

From \eqref{changingbasis}, it follows that $N \simeq M\big(\alpha,(0,a_0, a_+),\J_{a_0+a_+}\big)$. Moreover, since $K= \min\{a_-,a_0\} = a_0$, we deduce 
$$
M(\alpha,\boldsymbol a,\J_n)/\soc_1\big(M(\alpha,\boldsymbol a,\J_n)\big)\simeq M(\alpha,(0,K,a_+),\J_{K+a_+}).
$$

Next, assume that $a_-<a_0$. For $1\leq j\leq a_0-a_-$, let $D_j:=\C[h]\overline{\mathsf e}_j\subset N$,
where $\overline{\mathsf e}_j$ is the $j$th coordinate vector in the quotient space. It is clear that each $D_j$ is a submodule isomorphic to $B_+$. Set
$$
\mathbf T:=D_1\oplus\cdots\oplus D_{a_0-a_-}\subset N.
$$

Since the multiplicities of $A$ and $B_+$ in $\soc_1\big(M(\alpha,\boldsymbol a,\J_n)\big)$ have already been determined, we obtain
$$
M(\alpha,\boldsymbol a,\J_n)/\soc_1\big(M(\alpha,\boldsymbol a,\J_n)\big) \simeq N/\mathbf T.
$$

The induced $f$--matrix of $N/\mathbf T$ is
$$
\begin{pmatrix}
p(h)\I_{a_-} & p(h)c_{a_-}(h)\mathsf e_1^{\mathsf T}\\
0 & -\lambda_\alpha(h)\J_{a_+}
\end{pmatrix}.
$$

Once again, \eqref{changingbasis} yields $N/\mathbf T \simeq M\big(\alpha,(0,a_-,a_+),\J_{a_-+a_+}\big)$. Since $K= \min\{a_-,a_0\} = a_-$, we conclude that
$$
M(\alpha,\boldsymbol a,\J_n)/\soc_1\big(M(\alpha,\boldsymbol a,\J_n)\big) \simeq M\big(\alpha,(0,K,a_+),\J_{K+a_+}\big).
$$

Therefore, part (i) is proven.

\smallskip\noindent For part (ii), assume that $2\alpha-2\in\Z_{\geq 0}$. As above, the multiplicity of $A$ in $\soc_1\big(M(\alpha,\boldsymbol a,\J_n)\big)$ is $a_-$. By the same argument as in part (i),
together with the assumption $a_+ \leq 2a_-+a_0$, it follows that
$$
\dim\Hom_{\mathfrak{sl}(2)}\big(B_+,M(\alpha,\boldsymbol a,\J_n)\big)=(a_0-a_-)^{+},\qquad \dim\Hom_{\mathfrak{sl}(2)}\big(C,M(\alpha,\boldsymbol a,\J_n)\big)=0.
$$

Since $\alpha \in 1 + \tfrac{1}{2}\Z_{\geq 0}$, $B_+$ has a proper simple submodule, and $\soc(B_+) \simeq B_-$. Therefore,
$$
\soc_1\big(M(\alpha,\boldsymbol a,\J_n)\big) \simeq A^{\oplus a_-}\oplus B_-^{\oplus{(a_0-a_-)}^{+}} . 
$$

If $a_-\geq a_0$, then, as in part (i), we have $K=a_0$, and hence
$$
M(\alpha,\boldsymbol a,\J_n)/\soc_1\big(M(\alpha,\boldsymbol a,\J_n)\big)\simeq M(\alpha,(0,K,a_+),\J_{K+a_+}).
$$

Assume now that $a_-<a_0$. Then, for each $1\leq j\leq a_0-a_-$, $D_j$ is a submodule of $N$. Let
$$
S_j:=\soc(D_j)\simeq B_-,
\qquad
\mathbf S:=S_1\oplus\cdots\oplus S_{a_0-a_-}.
$$

Since the multiplicity of $B_+$ in $M(\alpha,\boldsymbol a,\J_n)$  is $(a_0-a_-)$, it follows that
$$
\soc_1\big(M(\alpha,\boldsymbol a,\J_n)\big)\simeq \mathbf U\oplus \mathbf S.
$$

Hence $M(\alpha,\boldsymbol a,\J_n)/\soc_1\big(M(\alpha,\boldsymbol a,\J_n)\big)\simeq N/ \mathbf S$. Under the identifications obtained after each quotient, the image of $S_j$ is the submodule $S=\soc(D)$ in the corresponding module $N_{a_0-j+1,a_+}$. Applying Lemma \ref{lem-resonant-step} repeatedly to $N \simeq N_{a_0,a_+}$, we obtain
$$
N/\mathbf S \simeq L(2\alpha-2)^{\oplus(a_0-a_-)}\oplus N_{a_-,a_+}.
$$

Finally, \eqref{changingbasis} implies $N_{a_-,a_+} \simeq M(\alpha,(0,a_-,a_+),\J_{a_-+a_+})$. Since $K=a_-$ in this case, we conclude that
$$
M(\alpha,\boldsymbol a,\J_n)/\soc_1\big(M(\alpha,\boldsymbol a,\J_n)\big)\simeq L(2\alpha-2)^{\oplus(a_0-a_-)}\oplus M(\alpha,(0,K,a_+),\J_{K+a_+}),
$$
which proves (ii).
\end{proof}

\begin{example}
We apply Proposition \ref{prop-soc1-and-quotient} to determine the socle layers of $M(\alpha,(1,2,6),\J_9)$ as follows
\begin{enumerate}
\item[(i)] If $2\alpha-2\notin\Z_{\geq 0}$, then the nonzero socle layers of $M(\alpha,(1,2,6),\J_9)$ are
\begin{eqnarray*}
\soc_3\big(M(\alpha,(1,2,6),\J_9)\big)/\soc_2\big(M(\alpha,(1,2,6),\J_9)\big) &\simeq& C,\\
\soc_2\big(M(\alpha,(1,2,6),\J_9)\big)/\soc_1\big(M(\alpha,(1,2,6),\J_9)\big) &\simeq& B_+\oplus C^{\oplus 3},\\
\soc_1\big(M(\alpha,(1,2,6),\J_9)\big) &\simeq& A\oplus B_+\oplus C^{\oplus 2}.\\
\end{eqnarray*}

\item[(ii)] If $2\alpha-2\in\Z_{\ge 0}$, then the nonzero socle layers of $M(\alpha,(1,2,6),\J_9)$ are
\begin{eqnarray*}
\soc_3\big(M(\alpha,(1,2,6),\J_9)\big)/\soc_2\big(M(\alpha,(1,2,6),\J_9)\big) &\simeq& L \oplus C,\\
\soc_2\big(M(\alpha,(1,2,6),\J_9)\big)/\soc_1\big(M(\alpha,(1,2,6),\J_9)\big) &\simeq& L \oplus B_-\oplus C^{\oplus 3},\\
\soc_1\big(M(\alpha,(1,2,6),\J_9)\big) &\simeq& A\oplus B_-\oplus C^{\oplus 2}.
\end{eqnarray*}
\end{enumerate}
\end{example}

\section{$U(\mathfrak h)$-Duality and Twisting by the Automorphism $\tau$} \label{sec-dual}
We recall two important functors: the $\mathcal U(\mathfrak h)$-dual, introduced in \cite{MP}, and the twisting functor by the automorphism $\tau$.
\subsection{$U(\mathfrak h)$-Duality }
We first recall the notion of $U(\mathfrak h)$-dual  in
the category $\mathcal M$.

\begin{definition}\label{duality}
For $M\in\mathcal M$, define the $U(\mathfrak h)$-dual of $M$ by
$$
M^{\vee_{\mathfrak h}}:=\Hom_{\C[h]}(M,\C[h]).
$$

It is equipped with the following $\mathfrak{sl}(2)$-action:
\begin{equation*}
(h\cdot g)(m):=h\,g(m)=g(h\cdot m), \qquad (e\cdot g)(m):=\sigma\bigl(g(f\cdot m)\bigr),
\qquad
(f\cdot g)(m):=\sigma^{-1}\bigl(g(e\cdot m)\bigr),
\end{equation*}
where $m\in M$ and $g\in M^{\vee_{\mathfrak h}}$.
\end{definition}

\begin{remark}\label{h-duality-functoriality}
The functor  $$
(\_)^{\vee_{\mathfrak h}}:\mathcal M^{\mathrm{op}}\to\mathcal M
$$
is a contravariant functor and not exact in general. 
\end{remark}

\begin{proposition}\label{dualstructure}
For $n \geq 1$, let $M(E,F)\in\mathcal M(n)$. Then
$$
M(E,F)^{\vee_{\mathfrak h}}
\simeq
M\bigl(\sigma(F^{\mathsf T}),\sigma^{-1}(E^{\mathsf T})\bigr).
$$
\end{proposition}

\begin{proof}
Recall that $\mathsf e_1,\dots, \mathsf e_n$ are the standard $\C[h]$-basis of $M(E,F) = \C[h]^n$, and let
$\varepsilon_1,\dots,\varepsilon_n$ be the dual basis of
$M(E,F)^{\vee_{\mathfrak h}}$. We identify
$\sum_i g_i\varepsilon_i$ with
$\boldsymbol g=(g_1,\dots,g_n)^{\mathsf T}$ via $$
\left( \sum_i g_i\varepsilon_i \right) (\boldsymbol v)=\boldsymbol g^{\mathsf T}\boldsymbol v
\qquad
(\boldsymbol v\in\C[h]^{\oplus n}).
$$

Using Definition~\ref{duality}, we obtain
\begin{align*}
(e\cdot \boldsymbol g)(\boldsymbol v)
&=
\sigma\bigl(\boldsymbol g^{\mathsf T}F\sigma^{-1}(\boldsymbol v)\bigr)
=
\bigl(\sigma(F^{\mathsf T})\sigma(\boldsymbol g)\bigr)^{\mathsf T}
\boldsymbol v,\\
(f\cdot \boldsymbol g)(\boldsymbol v)
&=
\sigma^{-1}\bigl(\boldsymbol g^{\mathsf T}E\sigma(\boldsymbol v)\bigr)
=
\bigl(\sigma^{-1}(E^{\mathsf T})\sigma^{-1}(\boldsymbol g)\bigr)^{\mathsf T}
\boldsymbol v.
\end{align*}

Therefore
$$
e\cdot\boldsymbol g=\sigma(F^{\mathsf T})\sigma(\boldsymbol g),
\qquad
f\cdot\boldsymbol g=\sigma^{-1}(E^{\mathsf T})\sigma^{-1}(\boldsymbol g),
$$
which proves the claim.
\end{proof}

\begin{corollary}\label{h-dual-normal-form}
For $n \geq 1$, let $\alpha\in\C$, $K(h)\in\GL_n(\C[h])$, and let
$\boldsymbol a=(a_-,a_0,a_+)\in\Z_{\geq 0}^3$ satisfy
$a_-+a_0+a_+=n$. Then
$$
M\bigl(\alpha,\boldsymbol a,K(h)\bigr)^{\vee_{\mathfrak h}}
\simeq
M\Bigl(\alpha,(a_+,-a_0,a_-),
-\T_n\bigl(K(h)^{-1}\bigr)^{\mathsf T}\T_n\Bigr).
$$
\end{corollary}

\begin{proof}
Recall that
$$
M\bigl(\alpha,\boldsymbol a,K(h)\bigr)
=
M\Bigl(
\sigma\bigl(K(h)^{-1}\overline P_{(\boldsymbol a,\alpha)}(h)\bigr),
P_{(\boldsymbol a,\alpha)}(h)K(h)
\Bigr).
$$

By Proposition~\ref{dualstructure}, the $f$-matrix of
$M\bigl(\alpha,\boldsymbol a,K(h)\bigr)^{\vee_{\mathfrak h}}$ is
$$
\sigma^{-1}
\left(
\sigma\bigl(K(h)^{-1}\overline P_{(\boldsymbol a,\alpha)}(h)\bigr)^{\mathsf T}
\right)
=
\overline P_{(\boldsymbol a,\alpha)}(h)
\bigl(K(h)^{-1}\bigr)^{\mathsf T} = -\overline P_{(\boldsymbol a,\alpha)}(h)
\bigl(-K(h)^{-1}\bigr)^{\mathsf T}.
$$

Since
$$
\T_n\bigl(-\overline P_{(\boldsymbol a,\alpha)}(h)\bigr)\T_n
=
P_{((a_+,-a_0,a_-),\alpha)}(h),
$$
the above $f$-matrix is $\sigma^{-1}$-similar to
$$
P_{((a_+,-a_0,a_-),\alpha)}(h)
\left(-\T_n\bigl(K(h)^{-1}\bigr)^{\mathsf T}\T_n\right).
$$

The result follows from Lemma~\ref{1st-isom-thm}.
\end{proof}

\begin{definition}
For $n \geq 1$, let $\alpha\in\C$, $K(h)\in\GL_n(\C[h])$, and let $\boldsymbol a=(a_-,a_0,a_+)\in\Z_{\geq 0}^3$ satisfy $a_-+a_0+a_+=n$. By $\mathfrak f\left(M\bigl(\alpha,\boldsymbol a,K(h)\bigr)\right)$ we denote the total multiplicity of all simple finite-dimensional 
submodules in the coherent family $\mathcal W\left(M\bigl(\alpha,\boldsymbol a,K(h)\bigr)\right)$. By 
Theorem~\ref{cohfamilydecomp},
$$
\mathfrak f\left(M\bigl(\alpha,\boldsymbol a,K(h)\bigr)\right)=
\begin{cases}
a_0+d, & \text{if } \alpha\in \frac{1}{2}\Z_{\leq0},\\
d, & \text{if } \alpha\in 1+\frac{1}{2}\Z_{\geq0},\\
0, & \text{otherwise},
\end{cases}
$$
where $d:= a_-+ a_+ -\rank\left(\boldsymbol M_{\boldsymbol a}\big(\mathcal K_{\alpha, \boldsymbol a, K} \big)\right)$.
\end{definition}

The following proposition gives a necessary and sufficient condition for the $\mathcal U(h)$-dual of a module in $\mathcal M(n)$ to be simple.
\begin{proposition}\label{prop-hdual-simple-defect}
For $n \geq 1$, let $\alpha\in\C$, $K(h)\in\GL_n(\C[h])$, and let $\boldsymbol a=(a_-,a_0,a_+)\in\Z_{\geq 0}^3$ satisfy $a_-+a_0+a_+=n$. Let $M: = M\bigl(\alpha,\boldsymbol a,K(h)\bigr)$ be simple. Then
$M^{\vee_{\mathfrak h}}$ is simple if and only if $ \mathfrak f\left(M\right)=0.$
\end{proposition}

\begin{proof}
By Proposition \refeq{fin-sub-coherentfam}, $ \mathfrak f\left(M\right)>0$ if and
only if there exists an exact sequence
$$
0\longrightarrow M\longrightarrow X\longrightarrow S\longrightarrow 0
$$
with $X\in\mathcal M(n)$ and $S$ is a nonzero finite-dimensional module. Observe that $M$, $X$ and $S$ have the same central character $(2\alpha-1)^2$. We now compare this condition with the simplicity of $M^{\vee_{\mathfrak h}}$. Assume first that such an exact sequence exists. Applying the functor  
$(\_)^{\vee_{\mathfrak h}}$ and this gives rise to its right derived functor in the category of modules that are finitely generated over $U(\mathfrak h) \simeq \C[h]$. In particular, we obtain
$$
0\longrightarrow X^{\vee_{\mathfrak h}}
\longrightarrow M^{\vee_{\mathfrak h}}
\longrightarrow \Ext_{\mathcal M_{\mathrm{fg}} }^1(S,\C[h])\longrightarrow 0.
$$

Since $\dim \Ext_{\mathcal M_{\mathrm{fg}} }^1(S,\C[h]) = \dim S$ 
and $\Ext_{\mathcal M_{\mathrm{fg}}}^1(S,\C[h])$ has the same central character as $S$, it follows that $\Ext_{\mathcal M_{\mathrm{fg}} }^1(S,\C[h]) \simeq S$ and hence $M^{\vee_{\mathfrak h}}$ is not simple.

\smallskip\noindent
Conversely, suppose that $M^{\vee_{\mathfrak h}}$ is not simple, and choose a nonzero proper submodule $N\subset M^{\vee_{\mathfrak h}}$. Since $M$ is
simple, $N$ must have full $\C[h]$-rank $n$; otherwise the annihilator
$$
N^\perp=\{m\in M: \varphi(m)=0\text{ for all }\varphi\in N\}
$$
would be a nonzero proper submodule of $M$. Hence $M^{\vee_{\mathfrak h}}/N$ is nonzero finite dimensional. Applying $(\_)^{\vee_{\mathfrak h}}$ and its right derived functor as above to
$$
0\longrightarrow N\longrightarrow M^{\vee_{\mathfrak h}}
\longrightarrow M^{\vee_{\mathfrak h}}/N\longrightarrow 0
$$
gives an exact sequence
$$
0\longrightarrow M\longrightarrow N^{\vee_{\mathfrak h}}
\longrightarrow
\Ext_{\mathcal M_{\mathrm{fg}} }^1(M^{\vee_{\mathfrak h}}/N,\C[h])\longrightarrow 0.
$$

Note that since $M\in\mathcal M(n)$, the natural evaluation map
$$
M\longrightarrow M^{\vee_{\mathfrak h}\vee_{\mathfrak h}}
$$
is an isomorphism of $\mathfrak{sl}(2)$-modules. The module
$N^{\vee_{\mathfrak h}}$ belongs to $\mathcal M(n)$, and the last term is
nonzero finite dimensional. Hence $M$ has a same-rank extension with nonzero
finite-dimensional quotient. Thus $M^{\vee_{\mathfrak h}}$ is not simple if and only if
$\mathfrak f(M)>0$. Equivalently,
$M^{\vee_{\mathfrak h}}$ is simple if and only if $\mathfrak f(M)=0$.
\end{proof}

\begin{example}
We now give an explicit example illustrating the above proposition. Let $\alpha\in 1+\tfrac12\Z_{\geq 0}$ and $\beta\in\C^*$. Then $M(\alpha,(0,-1,0),\beta)$ is simple, while
$$
M(\alpha,(0,-1,0),\beta)^{\vee_{\mathfrak h}}
\simeq
M\left(\alpha,(0,1,0),-\tfrac1\beta\right)
$$
 is not simple.
\end{example}

\subsection{Twisting by the automorphism $\tau$}
Let $\tau\in\Aut(\mathfrak{sl}(2))$ be defined by
$$
\tau(e)=-f,\qquad \tau(f)=-e,\qquad \tau(h)=-h.
$$

It is well-known that $(\_)^\tau$ is a covariant exact autoequivalence. This contrasts with
$U(\mathfrak h)$-duality, which is contravariant and not exact.

\begin{lemma}\label{tau-twist}
For $n \geq 1$, let $M(E,F)\in\mathcal M(n)$. Then
$$
M(E,F)^\tau\simeq M(-F(-h),-E(-h)).
$$
\end{lemma}

\begin{proof}
Write $\boldsymbol g=(g_1,\dots,g_n)^{\mathsf T}$ with $g_i\in\C[h]$. In the
twisted module $M(E,F)^\tau$, one has
$$
h\cdot\boldsymbol g=-h\boldsymbol g,\qquad
e\cdot\boldsymbol g=-F\sigma^{-1}(\boldsymbol g),\qquad
f\cdot\boldsymbol g=-E\sigma(\boldsymbol g).
$$

For convenience, set $\underline h:=-h$. Then $h$ acts by multiplication by $\underline h$, and
$$
\sigma(\boldsymbol g(\underline h))=\boldsymbol g(\underline h+1),
\qquad
\sigma^{-1}(\boldsymbol g(\underline h))=\boldsymbol g(\underline h-1).
$$

Therefore
$$
e\cdot\boldsymbol g(\underline h)
=
-F(-\underline h)\boldsymbol g(\underline h-1),
\qquad
f\cdot\boldsymbol g(\underline h)
=
-E(-\underline h)\boldsymbol g(\underline h+1),
$$
which proves the statement.
\end{proof}

\begin{proposition}\label{tau-normal-form}
For $n \geq 1$, let $\alpha\in\C$, $K(h)\in\GL_n(\C[h])$, and let
$\boldsymbol a=(a_-,a_0,a_+)\in\Z_{\geq 0}^3$ satisfy
$a_-+a_0+a_+=n$. Then
$$
M\bigl(\alpha,\boldsymbol a,K(h)\bigr)^\tau
\simeq
M\Bigl(\alpha,(a_+,a_0,a_-),
\T_nK(-h-2)^{-1}\T_n \D_{\boldsymbol a}\Bigr),
$$
where
$$
\D_{\boldsymbol a}:=
\diag\Bigl(
\underbrace{1,\ldots,1}_{a_+},
\underbrace{-1,\ldots,-1}_{a_0},
\underbrace{1,\ldots,1}_{a_-}
\Bigr).
$$
\end{proposition}

\begin{proof}
By Lemma~\ref{tau-twist}, the $f$-matrix of
$M\bigl(\alpha,\boldsymbol a,K(h)\bigr)^\tau$ is
$$
-E(-h)
=
K(-h-1)^{-1}
\bigl(-\overline P_{(\boldsymbol a,\alpha)}(-h-1)\bigr).
$$

Let
$$
P(h):=\D_{\boldsymbol a}\T_nK(-h-1),
\qquad
Q(h):=\T_n.
$$

A direct computation gives
$$
P(h)K(-h-1)^{-1}
\bigl(-\overline P_{(\boldsymbol a,\alpha)}(-h-1)\bigr)Q(h)
=
P_{((a_+,a_0,a_-),\alpha)}(h).
$$

Then by Corollary~\ref{howtofind-K(h)} we have
\begin{align*}
M\bigl(\alpha,\boldsymbol a,K(h)\bigr)^\tau
&\simeq
M\Bigl(\alpha,(a_+,a_0,a_-),Q(h)^{-1}P(h+1)^{-1}\Bigr)\\
&=
M\Bigl(\alpha,(a_+,a_0,a_-),
\T_nK(-h-2)^{-1}\T_n\D_{\boldsymbol a}\Bigr).
\end{align*}
\end{proof}

\section{Exponential modules and $U (\mathfrak h)$-free modules}
For $g(t) \in \C[t]$, write $g(t)= g(t)_{\mathrm{ev}} + g(t)_{\mathrm{odd}}$, where $g(t)_{\mathrm{ev}}$ (resp. $g(t)_{\mathrm{odd}}$) is the sum of the even-degree (resp. odd-degree) terms of $g(t)$. We call $g(t)$ \emph{even} (respectively, \emph{odd}) if $g(t)_{\mathrm{odd}}=0$ (respectively, if $g(t)_{\mathrm{ev}}=0$). We also denote
$$\C[t]_{\mathrm{ev}} = \Span\{t^{2k} :  k\geq 0\} \qquad \C[t]_{\mathrm{odd}}:=\Span\{t^{2k+1}:  k\geq 0\}$$
  
  The following lemma introduces several associative algebra homomorphisms that will be used throughout.
\begin{lemma}
Let $b\in\C$. Define
$$
\begin{aligned}
\Phi_1 :\qquad & e\mapsto t^2\partial + 2bt,\qquad f\mapsto -\partial,\qquad h\mapsto t\partial +b;\\
\Phi_2 :\qquad & e\mapsto -t\partial^2 + (2b-2)\partial,\qquad f\mapsto t,\qquad h\mapsto  -t\partial +(b-1);\\
\Phi_3 :\qquad & e\mapsto \tfrac{1}{2}t^2,\qquad f\mapsto -\tfrac{1}{2}\partial^2,\qquad h\mapsto \tfrac{1}{2}t\partial +\tfrac{1}{4}.\\
\end{aligned}
$$

Then, for each $i\in \{1,2,3\}$, the above assignment extends to an associative algebra homomorphism $\Phi_i: U(\mathfrak{sl}(2))\to\mathcal D(1)$.
\end{lemma}

\begin{proof}
Direct verification.
\end{proof}

\begin{definition}
For $b \in \C$, $i \in \{1,2\}$, and $g \in \C[t]$, let
$$
T_i(g,b) := \C[t]e^{g}, \qquad T_3(g) := \C[t]e^{g}.
$$

Then $T_1(g,b)$, $T_2(g,b)$, and $T_3(g)$ are $\mathfrak{sl}(2)$-modules via the homomorphisms $\Phi_1$, $\Phi_2$, and $\Phi_3$, respectively. Unless otherwise stated we will assume that $g = \sum_{i=1}^n a_it^i$, where $a_i \in \C$.
\end{definition}

\begin{remark}
Observe that
$$
\theta\circ \Phi_3=\Phi_3\circ \tau
\qquad\text{and}\qquad
\Phi_2=\theta\circ \Phi_1,
$$
where $\theta$ is the Fourier-transform automorphism of $\mathcal D(1)$, defined in Section \ref{coherentfam-section}, and $\tau$ is the automorphism of $\mathfrak{sl}(2)$, defined in Section~\ref{sec-dual}. With the $\tau$-twist and the Fourier-transform automorphism, we obtain six natural families of $\mathfrak{sl}(2)$-modules. It suffices to determine the simplicity criteria for the three families $T_1(g,b)$, $T_2(g,b)$, and $T_3(g)$.
\end{remark}

\begin{lemma}
For $i \in \{1,2,3\}$, the images of the Casimir element $c$ under $\Phi_i$ are
$$\Phi_1(c) =\Phi_2(c)=(2b-1)^2 ,  \quad\text{and}\quad \Phi_3(c) = \tfrac{1}{4} .$$
\end{lemma}

\begin{remark}
The exponential modules $E_{\pm}(g)$ for $\mathfrak{osp}(1|2)$ were studied in detail in~\cite{GN2}. We note that  $\mathfrak{osp}(1|2)_{\bar{0}} \simeq \mathfrak{sl}(2)$ and
$$T_3(g) \,\simeq\, \Res^{U(\mathfrak{osp}(1|2))}_{ U(\mathfrak{osp}(1|2)_{\bar{0}})}E_{+}(g), \qquad  T_3(g)^\tau \, \simeq\, \Res^{U(\mathfrak{osp}(1|2))}_{U(\mathfrak{osp}(1|2)_{\bar{0}})}E_{-}(g).$$
\end{remark}

The next proposition shows that, if $\deg g=n\geq 1$, then $T_1(g,b), T_2(g,b), T_3(g)$
are $U(\mathfrak h)$-free of rank $n$.

\begin{proposition} \label{U(h)-freeproof}
Let $g(t)$ be of degree $n\geq 1$ and $T\in \{ T_1(g,b), T_2(g,b), T_3(g)\}$.
Then
$$
T=\bigoplus_{p=0}^{n-1} \;U(\mathfrak h)\cdot(t^{p}e^{g}).
$$
as $ U(\mathfrak h)$-modules. In particular, $T\;\in \; \mathcal{M}(n)$.
\end{proposition}

\begin{proof}
We prove the statement for $T_1(g,b)$; the proofs for $T_2(g,b)$ and $T_3(g)$ are analogous.  Write $g(t)=\sum_{j=1}^n a_jt^j$ with $a_n\neq 0$.
Using $\Phi_1(h)=t\partial+b$, we have, for every $k\geq 0$,
$$
h\cdot(t^k e^g)
=
(k+b)t^k e^g+\sum_{j=1}^n j a_j t^{k+j}e^g .
$$

Taking $k=\ell-n$, with $\ell\geq n$, and solving for $t^\ell e^g$, we obtain
$$
t^\ell e^g
=
\frac{h-\ell+n-b}{na_n}\,t^{\ell-n}e^g
-
\sum_{j=1}^{n-1}\frac{j a_j}{na_n}\,t^{\ell-n+j}e^g .
$$

It follows by induction on $\ell$ that every $t^\ell e^g$ lies in the $U(\mathfrak h)$-span of
$e^g,\; te^g,\;\ldots,\; t^{n-1}e^g$.
Thus
$$
T_1(g,b)=\sum_{p=0}^{n-1}U(\mathfrak h)\cdot(t^p e^g).
$$

It remains to show that this sum is direct. From the formula for the action of $h$, we have
$$
h^d\cdot t^p e^g=(na_n)^d t^{p+nd}e^g+\text{terms of smaller $t$-degree}.
$$

Suppose that
$$
\sum_{p=0}^{n-1} q_p(h)t^p e^g=0,
\qquad q_p(h)\in \mathbb C[h].
$$

If some $q_p$ is nonzero, choose $r$ such that $r+n\deg q_r$ is maximal among all $p+n\deg q_p$ with $q_p\neq 0$. Such $r$ is unique, since the integers $p+n\deg q_p$, for $0\leq p\leq n-1$, have distinct residues modulo $n$. The summand $q_r(h)t^r e^g$ then contributes a nonzero multiple of $t^{r+n\deg q_r}e^g$,
while all other summands have strictly smaller $t$-degree. This is impossible. Hence all $q_p=0$, and the sum is direct.
\end{proof}

\subsection{Realizing exponential modules as $M\big(\alpha, {\boldsymbol{a}}, K(h)\big)$} The following result provides an explicit realization of $T_i(g,b)$, for $i\in\{1,2\}$, as a module of the form $M(\alpha,\boldsymbol a,K(h))$.
\begin{proposition} \label{realizationT-1}
Let $g(t)=\sum_{i=1}^n a_i t^i$, where $n\geq 2$ and $a_n\neq 0$. Then
$$
T_1(g,b)\simeq M\bigl(b,(1,n-1,0),K_1\bigr),
\qquad
T_2(g,b)\simeq M\bigl(b,(n-1,1,0),K_2\bigr),
$$
where $K_1$ and $K_2$ are the following matrices:
$$
K_1:=
\begin{pmatrix}
 -a_1 & -1 & 0 & \cdots & 0 \\
 -2a_2 & 0 & -1 & \cdots & 0 \\
 \vdots & \vdots & \ddots & \ddots & \vdots \\
 -(n-1)a_{n-1} & 0 & \cdots & 0 & -1 \\
 -na_n & 0 & \cdots & 0 & 0
\end{pmatrix},
\qquad
K_2:=
\begin{pmatrix}
0      & 0      & \cdots & -\tfrac{a_1}{na_n}      & -\tfrac{1}{na_n} \\
1      & 0      & \cdots & -\tfrac{2a_2}{na_n}      & 0              \\
0      & 1      & \ddots & \vdots & \vdots         \\
\vdots & \vdots & \ddots & -\tfrac{(n-1)a_{n-1}}{na_n}       & 0            \\
0      & 0      & \cdots & 1      & 0
\end{pmatrix}.
$$
\end{proposition}

\begin{proof}
Set
$$
v_i:=t^i e^g,\qquad 0\leq i\leq n-1.
$$

By Proposition~\ref{U(h)-freeproof}, the elements $v_0,\ldots,v_{n-1}$ form a $U(\mathfrak h)$-basis of both $T_1(g,b)$ and $T_2(g,b)$. Thus it remains to compute the matrices $F_1, F_2$ of the $f$-action on $T_1(g,b)$, $T_2(g,b)$,  with respect to this basis and compare it to the one from from Corollary~\ref{howtofind-K(h)}. 

We first consider $T_1(g,b)$. Since $\Phi_1(f) =-\partial$, we have
$$
f v_0=-g'(t)e^g
=-\sum_{j=1}^n j a_j v_{j-1}, \qquad f v_i
=-i v_{i-1}-g'(t)t^i e^g
=-(h-b+1)v_{i-1} \;\; (i\geq 1).
$$

Hence the matrix $F_1$ of $f$ with the convention $f\boldsymbol v=F_1\sigma^{-1}(\boldsymbol v)$ is
$$
F_1=
\begin{pmatrix}
-a_1      & -(h-b+1)     & 0      & \cdots & 0     \\
-2a_2      & 0      & -(h-b+1)     & \cdots & 0     \\
\vdots & \vdots & \ddots & \ddots & \vdots \\
-(n-1)a_{n-1}      & 0      & \cdots & 0      & -(h-b+1)    \\
-na_n     & 0      & \cdots & 0      & 0
\end{pmatrix}.
$$

Let 
$$
L_1=
\begin{pmatrix}
0      & 0      & \cdots & 0      & -\tfrac{1}{na_n} \\
-1      & 0      & \cdots & 0      & \tfrac{a_1}{na_n}              \\
0      & -1      & \ddots & \vdots & \vdots         \\
\vdots & \vdots & \ddots & 0      & \tfrac{(n-2)a_{n-2}}{na_n}             \\
0      & 0      & \cdots & -1      & \tfrac{(n-1)a_{n-1}}{na_n}
\end{pmatrix}.
$$

A direct calculation gives
$$
L_1F_1=P_{((1,n-1,0),b)}(h).
$$

Therefore, by Corollary~\ref{howtofind-K(h)}, the corresponding $K$-matrix is $L_1^{-1}$, which is the displayed matrix $K_1$ in the statement.  This proves the part for $T_1(g,b)$.

Now consider $T_2(g,b)$. Here $\Phi_2(f)=t$. From the action of $h$ on $e^g$, we have
$$
(h-b+1)e^g=-\sum_{j=1}^n j a_j t^j e^g,
$$
and therefore
$$
f v_i=v_{i+1}\qquad (0\leq i\leq n-2), \qquad f v_{n-1}
=
-\frac{1}{na_n}
\left((h-b+1)v_0+\sum_{j=1}^{n-1}j a_j v_j\right).
$$

Thus the matrix $F_2$ of $f$ is
$$
F_2=
\begin{pmatrix}
0      & 0      & \cdots & 0      & -\tfrac{1}{na_n}(h-b+1) \\
1      & 0      & \cdots & 0      & -\tfrac{a_1}{na_n}              \\
0      & 1      & \ddots & \vdots & \vdots         \\
\vdots & \vdots & \ddots & 0      & -\tfrac{(n-2)a_{n-2}}{na_n}             \\
0      & 0      & \cdots & 1      & -\tfrac{(n-1)a_{n-1}}{na_n}
\end{pmatrix}.
$$

Let
$$
L_2=
\begin{pmatrix}
0 & 1 & 0 & \cdots & 0 \\
0 & 0 & 1 & \cdots & 0 \\
\vdots & \vdots & \ddots & \ddots & \vdots \\
0 & 0 & \cdots & 0 & 1 \\
-na_n & 0 & \cdots & 0 & 0
\end{pmatrix},
\qquad
R_2=
\begin{pmatrix}
1 & 0 & 0 & \cdots & \tfrac{a_1}{na_n} \\
0 & 1 & 0 & \cdots & \tfrac{2a_2}{na_n} \\
\vdots & \vdots & \ddots & \ddots & \vdots \\
0 & 0 & \cdots & 1 & \tfrac{(n-1)a_{n-1}}{na_n} \\
0 & 0 & \cdots & 0 & 1
\end{pmatrix}.
$$

Again, a direct calculation gives
$$
L_2F_2R_2=P_{((n-1,1,0),b)}(h).
$$

Thus $\SNF \left(F_{2}\right) =P_{((n-1,1,0) ,b)}(h)$ and by Corollary~\ref{howtofind-K(h)} we see that  the corresponding $K$-matrix is $ R_2^{-1}L_2^{-1}$, which is the displayed matrix $K_2$. This proves the statement for $T_2(g,b)$.
\end{proof}

\begin{remark} \label{rank1T1}
In the case where $g=a_1t$ with $a_1 \in \C^{*}$, we have
$$T_1(a_1t,b) \simeq M\big(b, (1,0,0), -a_1 \big),\qquad T_2(a_1t,b) \simeq M\big(b, (0,1,0), -\tfrac{1}{a_1} \big).$$

Using the same argument, one also obtains $T_3(a_1t) \simeq M\big(\tfrac{1}{4}, (1,0,0), -\tfrac{1}{2}a_1^2\big)$.
\end{remark}

\begin{proposition} \label{matrixformT3}
Let $g(t)=\sum_{i=1}^n a_i t^i,\; n\geq 2,\; a_n\neq 0$,
and set $u(h):=2h+\tfrac12$. Then
$$
T_3(g)\simeq
M\left(\tfrac14,(1,0,n-1),R(h)^{-1}L(h+1)^{-1}\right),
$$
where $L(h),R(h)\in \GL_n(\mathbb C[h])$ are defined as follows.

If $a_1\neq 0$, then
$$
R(h)=
\begin{pmatrix}
1 & -\tfrac{1}{a_1}u(h+\tfrac12) & 0 & \cdots & 0\\
0 & 1 & 0 & \cdots & 0\\
0 & 0 & -1 & \ddots & \vdots\\
\vdots & \vdots & \ddots & \ddots & 0\\
0 & 0 & \cdots & 0 & -1
\end{pmatrix},
$$
and
$$
L(h)=
\begin{pmatrix}
0 & 0 & 0 & \cdots & 0 & 0 & -\tfrac{2}{na_1a_n} \\
0 & 0 & 0 & \cdots & 0 & \tfrac{a_1}{2na_n} &
-\tfrac{(n-1)a_1a_{n-1}+na_nu(h)}{2n^2a_n^2} \\
\tfrac{1}{2} & 0 & 0 & \cdots & 0 & -\tfrac{2a_2}{2na_n} &
\tfrac{-na_1a_n+2(n-1)a_2a_{n-1}}{2n^2a_n^2} \\
0 & \tfrac{1}{2} & 0 & \cdots & 0 & -\tfrac{3a_3}{2na_n} &
\tfrac{-2na_2a_n+3(n-1)a_3a_{n-1}}{2n^2a_n^2} \\
0 & 0 & \tfrac{1}{2} & \ddots & \vdots & -\tfrac{4a_4}{2na_n} &
\tfrac{-3na_3a_n+4(n-1)a_4a_{n-1}}{2n^2a_n^2} \\
\vdots & \vdots & \ddots & \ddots & 0 & \vdots & \vdots \\
0 & 0 & \cdots & \tfrac{1}{2} & 0 & -\tfrac{(n-2)a_{n-2}}{2na_n} &
\tfrac{-(n-3)na_{n-3}a_n+(n-2)(n-1)a_{n-2}a_{n-1}}{2n^2a_n^2} \\
0 & 0 & \cdots & 0 & \tfrac{1}{2} & -\tfrac{(n-1)a_{n-1}}{2na_n} &
\tfrac{-(n-2)na_{n-2}a_n+(n-1)^2a_{n-1}^2}{2n^2a_n^2}
\end{pmatrix}.
$$

If $a_1=0$, then
$$
R(h)=
\begin{pmatrix}
1 & \tfrac{1}{2}u(h+\tfrac12) & 0 & \cdots & 0\\
1 & \tfrac{1}{2}u(h) & 0 & \cdots & 0\\
0 & 0 & 1 & \ddots & \vdots\\
\vdots & \vdots & \ddots & \ddots & 0\\
0 & 0 & \cdots & 0 & 1
\end{pmatrix},
$$
and
$$
L(h)=
\begin{pmatrix}
0 & 0 & 0 & \cdots & 0 & \tfrac{2}{na_n} &
-\tfrac{2((n-1)a_{n-1}+na_n)}{n^2a_n^2} \\
0 & 0 & 0 & \cdots & 0 & -\tfrac{1}{na_n}u(h+\tfrac12) &
\tfrac{(n-1)a_{n-1}u(h+\tfrac12)+na_nu(h)}{n^2a_n^2} \\
-\tfrac{1}{2} & 0 & 0 & \cdots & 0 & \tfrac{2a_2}{2na_n} &
-\tfrac{(n-1)a_2a_{n-1}}{n^2a_n^2} \\
0 & -\tfrac{1}{2} & 0 & \cdots & 0 & \tfrac{3a_3}{2na_n} &
\tfrac{2na_2a_n-3(n-1)a_3a_{n-1}}{2n^2a_n^2} \\
0 & 0 & -\tfrac{1}{2} & \ddots & \vdots & \tfrac{4a_4}{2na_n} &
\tfrac{3na_3a_n-4(n-1)a_4a_{n-1}}{2n^2a_n^2} \\
\vdots & \vdots & \ddots & \ddots & 0 & \vdots & \vdots \\
0 & 0 & \cdots & -\tfrac{1}{2} & 0 & \tfrac{(n-2)a_{n-2}}{2na_n} &
\tfrac{(n-3)na_{n-3}a_n-(n-2)(n-1)a_{n-2}a_{n-1}}{2n^2a_n^2} \\
0 & 0 & \cdots & 0 & -\tfrac{1}{2} & \tfrac{(n-1)a_{n-1}}{2na_n} &
\tfrac{(n-2)na_{n-2}a_n-(n-1)^2a_{n-1}^2}{2n^2a_n^2}
\end{pmatrix}.
$$
\end{proposition}

\begin{proof} Using the idea from the last proof, let $F_3 $ be the matrix of the $f$-action on $T_3(g)$ with respect to the $U(\mathfrak h)$-basis $\{e^g,\ te^g,\ \ldots,\ t^{n-1}e^g\}$. By Theorem~6.3 in~\cite{GN2},
$$
F=-X_-(h)X_-\left(h+\tfrac12\right),
$$
where
$$
X_{-}(h)=\tfrac{1}{\sqrt{2}}
\begin{pmatrix}
 a_1 &  u(h)  & 0 & \cdots & 0 \\
 2a_2 & 0 &  u(h)  & \cdots & 0 \\
 \vdots & \vdots & \ddots & \ddots & \vdots \\
 (n-1)a_{n-1} & 0 & \cdots & 0 &  u(h)  \\
 na_n & 0 & \cdots & 0 & 0
\end{pmatrix}.
$$

Since
$$
\det(F)=(-1)^n2^{n-2}n^2a_n^2
\left(h+\tfrac14\right)^{n-1}
\left(h+\tfrac34\right)^{n-1},
$$
with standard divisibility reasoning, it follows that $\SNF(F)=P_{((1,0,n-1),1/4)}(h)$.

For the matrices $L(h)$ and $R(h)$ defined above, one checks directly, in both cases $a_1\neq0$ and $a_1=0$, that
$$
L(h)FR(h)=P_{((1,0,n-1),1/4)}(h).
$$
Therefore, by Corollary~\ref{howtofind-K(h)}, we obtain
$$
T_3(g)\simeq
M\left(\tfrac14,(1,0,n-1),R(h)^{-1}L(h+1)^{-1}\right).
$$
\end{proof}

\begin{remark}
For reference, when $n=2$, the matrices in Proposition~\ref{matrixformT3}
specialize as follows. If $a_1\neq 0$, then
$$
F_3=-\tfrac12
\begin{pmatrix}
a_1^2+2a_2u(h) & a_1u(h+\tfrac12)\\
2a_1a_2 & 2a_2u(h+\tfrac12)
\end{pmatrix},
\;
L(h)=
\begin{pmatrix}
0 & -\tfrac{1}{a_1a_2}\\
\tfrac{a_1}{4a_2} & -\tfrac{a_1^2+2a_2u(h)}{8a_2^2}
\end{pmatrix},
\;
R(h)=
\begin{pmatrix}
1 & -\tfrac{u(h+\tfrac12)}{a_1}\\
0 & 1
\end{pmatrix}.
$$

If $a_1=0$, then
$$
F_3\hskip-4pt =\hskip-4pt -
\begin{pmatrix}
a_2u(h) & 0\\
0 & a_2u(h+\tfrac12)
\end{pmatrix},
\; 
L(h)\hskip -3pt =  \hskip-4pt \begin{pmatrix} \frac{1}{a_2} & -\tfrac{a_1 +2a_2}{2a_2^2}\\ -\tfrac{1}{2a_2}u(h+\frac{1}{2}) & \tfrac{a_1u(h+\frac{1}{2}) +2a_2u(h)}{4a_2^2} \end{pmatrix},
\; 
R(h)\hskip-4pt =\hskip-4pt 
\begin{pmatrix}
1 & \tfrac12u(h+\tfrac12)\\
1 & \tfrac12u(h)
\end{pmatrix}.
$$
\end{remark}

\begin{proposition} \label{decompT-3}
Let $m\geq 1$ and  $g(t)=\sum_{j=1}^m a_{2j}t^{2j}$ with $a_{2m}\neq 0$.
Then
$$
T_3(g)=T_3(g)^{\mathrm{ev}}\oplus T_3(g)^{\mathrm{odd}},
$$
where $T_3(g)^{\mathrm{ev}}:=\mathbb C[t]_{\mathrm{ev}}e^g$ and $T_3(g)^{\mathrm{odd}}:=\mathbb C[t]_{\mathrm{odd}}e^g$.
Moreover,
$$
T_3(g)^{\mathrm{ev}}
\simeq
M\left(\tfrac14,(0,-1,m-1),K_3\right),
\qquad
T_3(g)^{\mathrm{odd}}
\simeq
M\left(\tfrac14,(0,1,m-1),K_3\right),
$$
where
$$
K_3=
\begin{pmatrix}
-2 a_2 & -2 & 0 & \cdots & 0\\
-4 a_4 & 0 & -2 & \cdots & 0\\
\vdots & \vdots & \ddots & \ddots & \vdots\\
-2(m-1)a_{2(m-1)} & 0 & \cdots & 0 & -2\\
-2m a_{2m} & 0 & \cdots & 0 & 0
\end{pmatrix}.
$$

For $m=1$, $K_3=(-2a_2)$.
\end{proposition}

\begin{proof}
Since $g$ is even, the operators $\Phi_3(e)=\frac12t^2$, $\Phi_3(f)=-\frac12\partial^2$, and $\Phi_3(h)=\frac12t\partial+\frac14$ preserve $\mathbb C[t]_{\mathrm{ev}}e^g$ and $\mathbb C[t]_{\mathrm{odd}}e^g$.
Hence $T_3(g)^{\mathrm{ev}}$ and $T_3(g)^{\mathrm{odd}}$
are $\mathfrak{sl}(2)$-submodules, and
$$
T_3(g)=T_3(g)^{\mathrm{ev}}\oplus T_3(g)^{\mathrm{odd}}.
$$

Consider the ordered $U(\mathfrak h)$-bases
$\{e^g,t^2e^g,\ldots,t^{2m-2}e^g\}$ of  $T_3(g)^{\mathrm{ev}}$, and \\
$\{ te^g,t^3e^g,\ldots,t^{2m-1}e^g \}$
of $T_3(g)^{\mathrm{odd}}$. With respect to these bases, the corresponding matrices of the $f$-action are
$$
F_{\mathrm{ev}}=
\begin{pmatrix}
-a_2u(h) & 2\lambda_{1/4}  & 0 & \cdots & 0 \\
-2 a_4u(h) & 0 & 2\lambda_{1/4}  & \cdots & 0 \\
\vdots & \vdots & \ddots & \ddots & \vdots \\
-(m-1)a_{2(m-1)}u(h) & 0 & \cdots & 0 & 2\lambda_{1/4}  \\
-m a_{2m}u(h) & 0 & \cdots & 0 & 0
\end{pmatrix},
$$
and the matrix $F_{\mathrm{odd}}$ is obtained from $F_{\mathrm{ev}}$ by replacing $u(h)$ by $u(h+\tfrac12)$.

Let
$$
L_3=
\begin{pmatrix}
0 & 0 & \cdots & 0 & -\tfrac{1}{2ma_{2m}} \\
-\frac12 & 0 & \cdots & 0 & \tfrac{a_2}{2ma_{2m}} \\
0 & -\frac12 & \ddots & \vdots & \vdots \\
\vdots & \vdots & \ddots & 0 & \tfrac{(m-2)a_{2(m-2)}}{2ma_{2m}} \\
0 & 0 & \cdots & -\frac12 & \tfrac{(m-1)a_{2(m-1)}}{2ma_{2m}}
\end{pmatrix}.
$$

A direct calculation gives
$$
L_3F_{\mathrm{ev}}=P_{((0,-1,m-1),1/4)}(h),
\qquad
L_3F_{\mathrm{odd}}=P_{((0,1,m-1),1/4)}(h).
$$

Therefore, by Corollary~\ref{howtofind-K(h)}, the two claimed realizations with $K_3=L_3^{-1}$ follow.
\end{proof}

\subsection{Simplicity criteria}
In this subsection we prove necessary and sufficient conditions for the three families of exponential modules to be simple. It is worth noting that the proofs use a standard linear algebra observation, Lemma~\ref{lemforT3}, which makes the arguments shorter and may be useful in other related cases.

\begin{theorem} \label{simplicityT-2}
The module $T_2(g,b)$ is simple if and only if $b\notin 1+\tfrac12\mathbb Z_{\geq 0}$.
\end{theorem}

\begin{proof}
We first prove the ``if'' direction. Define
$$
\mathcal L_2
:=
e+f\bigl(g''(f)+(g'(f))^2\bigr)
-(2b-2)g'(f)
-2g'(f)\bigl(h+fg'(f)-b+1\bigr).
$$

A direct computation gives
$$
\mathcal L_2\cdot t^k e^g
=
k(2b-k-1)t^{k-1}e^g,
\qquad k\geq 0.
$$

Let $N$ be a nonzero submodule of $T_2(g,b)$, and choose 
$ 0\neq v=\sum_{i=0}^m b_i t^i e^g\in N$ with $b_m\neq 0$. If $m=0$, then $e^g\in N$.
If $m\geq 1$, then $m(2b-m-1)\neq 0$, since
$b\notin 1+\tfrac12\mathbb Z_{\geq0}$. Hence $\mathcal L_2v$ has degree $m-1$. Repeating this reduction, we obtain $e^g\in N$. Since $\Phi_2(f) =t$, we have $t^k e^g=f^k e^g\in N$ for all $k\geq 0$. Thus $N=T_2(g,b)$, and $T_2(g,b)$ is simple.

\smallskip\noindent
Conversely, suppose $b\in 1+\tfrac12\mathbb Z_{\geq0}$, and set $m_0:=2b-1\in\mathbb Z_{\geq1}$. Let
$$
N_{m_0}:=\Span\{t^p e^g\mid p\geq m_0\}.
$$

This is a nonzero proper subspace of $T_2(g,b)$. We show that it is a submodule. Since $f$ acts by multiplication by $t$, the subspace $N_{m_0}$ is stable under $f$. It remains to check stability under $e$.  For this we use
$$
e\cdot t^p e^g
=
p(2b-p-1)t^{p-1}e^g
-
\Bigl(2p g' + t(g''+(g')^2)-(2b-2)g'\Bigr)t^p e^g .
$$

If $p\geq m_0+1$, then all terms belong to $N_{m_0}$. If $p=m_0$, the only possible term of degree $m_0-1$ has coefficient
$$
m_0(2b-m_0-1)=0.
$$

Thus $eN_{m_0}\subseteq N_{m_0}$. Therefore $N_{m_0}$ is a proper nonzero submodule of $T_2(g,b)$, and $T_2(g,b)$ is not simple.
\end{proof}

\begin{remark}
If $2b-1=m_0\in\mathbb Z_{\geq1}$ and $\deg g=n\geq1$, then the submodule $N_{m_0}$ from the proof belongs to $\mathcal M(n)$. Indeed, the same argument as in Proposition~\ref{U(h)-freeproof} gives
$$
N_{m_0}
=
\bigoplus_{p=m_0}^{m_0+n-1}
U(\mathfrak h)\cdot t^p e^g.
$$
\end{remark}


By Remark \ref{rank1T1} and Theorem \ref{simplicity-rank1}, the module $T_1(a_1t,b)$ is simple for all $a_1\in\C^*$ and $b\in\C$. Accordingly, in this subsection, we restrict our attention to $T_1(g,b)$ under the assumption that $\deg g\geq 2$.

\begin{lemma}\label{lemforT3}
Let $V$ be a subspace of $\mathbb C[t]$. Suppose that there exists
$N \geq 0 $ such that, for every $k\geq0$, the subspace $V$
contains a polynomial of degree $N+k$. Then $\mathbb C[t]/V$ is
finite dimensional.
\end{lemma}

\begin{proof}
This follows from a standard linear algebra argument, so we omit the proof.
\end{proof}

\begin{lemma} \label{T1new}
Let $g\in\mathbb C[t]$ with $\deg g\geq2$. If $V$ is a nonzero proper
submodule of $T_1(g,b)$, then $T_1(g,b)/V$ is finite dimensional.
\end{lemma}

\begin{proof}
Choose $0\neq p(t)e^g\in V$, with $p(t)\in\mathbb C[t]$, and set
$d:=\deg g$. Since $T_1(g,b)$ is defined by $\Phi_1$, we have
$$
f\cdot pe^g=-(p'+pg')e^g,\qquad
(h-b)\cdot pe^g=(p'+pg')te^g .
$$

With the convention that $\deg(ue^g) := \deg(u)$, for every nonzero $q\in\mathbb C[t]$,
$$
\deg(f\cdot qe^g)=\deg q+d-1,\qquad
\deg((h-b)\cdot qe^g)=\deg q+d.
$$

Indeed, these degrees come from the leading term of $qg'$.

Since $V$ is a submodule, repeated applications of $f$ and $h-b$ to
$p(t)e^g$ show that $V$ contains elements of degrees
$$
\deg p+r(d-1)+sd,\qquad r,s\in\mathbb Z_{\geq0}.
$$

Because $\gcd(d-1,d)=1$, every sufficiently large integer is of the form
$r(d-1)+sd$ with $r,s\geq0$. Hence, for every sufficiently large integer $D$, there exists
$p_D(t)e^g\in V$ with $\deg p_D=D$. Now Lemma~\ref{lemforT3}
implies that $T_1(g,b)/V$ is finite dimensional.
\end{proof}
\begin{corollary}\label{nolowersub(1)}
Let $n=\deg g\geq 2$. The module $T_1(g,b)$ has no nonzero
$U(\mathfrak h)$-free submodules of rank strictly less than $n$.
\end{corollary}

\begin{proof}
Suppose that $V$ is a nonzero $U(\mathfrak h)$-free submodule of
$T_1(g,b)$ of rank $r<n$. Since $T_1(g,b)$ is $U(\mathfrak h)$-free
of rank $n$, the quotient $T_1(g,b)/V$ has positive $U(\mathfrak h)$-rank
and is therefore infinite dimensional. This contradicts Lemma~\ref{T1new}.
\end{proof}

\begin{theorem} \label{simplicityT-1}
Let $\deg g\geq 2$. The module $T_1(g,b)$ is simple if and only if $b\notin 1+\tfrac12\mathbb Z_{\geq0}$.
\end{theorem}

\begin{proof}
By Proposition~\ref{realizationT-1}, the module $T_1(g,b)$ is realized as
$M\bigl(b,(1,n-1,0),K_1\bigr)$, where $n=\deg g$. The result now follows
from Proposition~\ref{generalcaseprop}, together with
Corollary~\ref{nolowersub(1)}, which rules out proper submodules of smaller
$U(\mathfrak h)$-rank.
\end{proof}

\begin{proposition}
Let $k\in\mathbb Z_{\geq0}$, and let $g_1,g_2\in\mathbb C[t]$ with
$\deg g_1=n\geq2$ and $\deg g_2\geq1$. Then the following short exact
sequences hold:
$$
0\longrightarrow T_1\left(g_1,-\tfrac{k}{2}\right)
\longrightarrow T_1\left(g_1,1+\tfrac{k}{2}\right)
\longrightarrow L(k)^{\oplus(n-1)}
\longrightarrow 0,
$$
$$
0\longrightarrow T_2\left(g_2,-\tfrac{k}{2}\right)
\longrightarrow T_2\left(g_2,1+\tfrac{k}{2}\right)
\longrightarrow L(k)
\longrightarrow 0.
$$
\end{proposition}

\begin{proof}
The first sequence follows from Propositions~\ref{SESgeneralcase} and
\ref{realizationT-1}. The second follows in the same way, using
Proposition~\ref{realizationT-1} when $\deg g_2\geq2$, and
Remark~\ref{rank1T1} when $\deg g_2=1$.
\end{proof}

\begin{remark}
For every $k\in\mathbb Z_{\geq0}$, we have
$-\tfrac{k}{2}\notin 1+\tfrac12\mathbb Z_{\geq0}$. Hence
Theorems~\ref{simplicityT-1} and~\ref{simplicityT-2} imply that
$T_1(g_1,-\tfrac{k}{2})$ and $T_2(g_2,-\tfrac{k}{2})$ are simple.
\end{remark}

\begin{proposition} \label{T3-simplicity}
If $g$ is a nonconstant even polynomial, then
$T_3(g)^{\mathrm{ev}}$ and $T_3(g)^{\mathrm{odd}}$ are simple.
\end{proposition}

\begin{proof}
We prove the statement for both submodules at once. Let
$$
S\subset T_3(g)^{\epsilon},\qquad \epsilon\in\{\mathrm{ev},\mathrm{odd}\},
$$
be a nonzero submodule, and choose $0\neq p(t)e^g\in S$, with
$p(t)\in\mathbb C[t]_{\epsilon}$, such that $\deg p$ is minimal.

Since $\Phi_3(e)=\frac12t^2$, the submodule $S$ contains $q(t^2)p(t)e^g$ for every
$q\in\mathbb C[t]$. Also,
$$
\left(2h-\tfrac12\right)\cdot p(t)e^g
=
\bigl(tp'(t)+tg'(t)p(t)\bigr)e^g .
$$

Because $g$ is even, $tg'(t)\in\mathbb C[t^2]$. Hence
$tg'(t)p(t)e^g\in S$, and therefore $tp'(t)e^g\in S$.

Next,
$$
f\cdot p(t)e^g
=
-\tfrac12\bigl(p''(t)+2g'(t)p'(t)+((g'(t))^2+g''(t))p(t)\bigr)e^g .
$$

Since $(g'(t))^2+g''(t)\in\mathbb C[t^2]$, the last term belongs to $S$.
Moreover, $g'(t)/t\in\mathbb C[t^2]$, and we have already shown that
$tp'(t)e^g\in S$; hence $2g'(t)p'(t)e^g\in S$. It follows that
$p''(t)e^g\in S$.

If $p''\neq0$, then $p''$ has the same parity as $p$ and smaller degree,
contradicting the minimal choice of $p$. Thus $p''=0$. In the even case,
$p$ is constant, so $e^g\in S$. In the odd case, $p$ is a scalar multiple
of $t$, so $te^g\in S$. Applying powers of $e$, we obtain
$T_3(g)^{\mathrm{ev}}\subseteq S$ in the even case and
$T_3(g)^{\mathrm{odd}}\subseteq S$ in the odd case. Hence
$S=T_3(g)^\epsilon$, and both modules are simple.
\end{proof}

\begin{theorem}
The module $T_3(g)$ is simple if and only if $g$ is not an even polynomial,
equivalently, if $g_{\mathrm{odd}}\neq 0$.
\end{theorem}

\begin{proof}
If $g$ is even, the actions of $e,f,h$ preserve
$\mathbb C[t]_{\mathrm{ev}}e^g$ and $\mathbb C[t]_{\mathrm{odd}}e^g$. Hence these are nonzero proper submodules of $T_3(g)$.
Hence $T_3(g)$ is not simple.

Conversely, suppose that $g_{\mathrm{odd}}\neq0$. Let $V$ be a nonzero
proper submodule of $T_3(g)$, and choose $0\neq p(t)e^g\in V$. Since
$\Phi_3(e)=\frac12t^2$, we have $t^{2k}p(t)e^g\in V$ for $k\geq0$.
Also,
$$
2\left(h-\tfrac14\right)\cdot p(t)e^g
=
\bigl(tp'(t)+tp(t)g'_{\mathrm{ev}}+tp(t)g'_{\mathrm{odd}}\bigr)e^g .
$$

Since $tg'_{\mathrm{ev}}\in\mathbb C[t^2]$, the term
$tp(t)g'_{\mathrm{ev}}e^g$ belongs to $V$. Therefore
$$
\bigl(tp'(t)+tp(t)g'_{\mathrm{odd}}\bigr)e^g\in V.
$$

The same argument, with $p$ replaced by any $q$ such that $q(t)e^g\in V$, shows that the second operation also preserves $V$.
 Thus, starting from
$p(t)e^g$, we may repeatedly apply the two operations
$$
q(t)\mapsto t^2q(t),\qquad
q(t)\mapsto tq'(t)+tq(t)g'_{\mathrm{odd}}.
$$

The first operation raises the degree by $2$. If $r=\deg g_{\mathrm{odd}}$,
then the second operation raises the degree by $r$, since its leading term
comes from $tq(t)g'_{\mathrm{odd}}$. As $r$ is odd, the integers $2$ and
$r$ are relatively prime. Hence, for every sufficiently large integer $D$,
the submodule $V$ contains an element $p_D(t)e^g$ with $\deg p_D=D$.
Identifying $T_3(g)=\mathbb C[t]e^g$ with $\mathbb C[t]$, Lemma~\ref{lemforT3}
implies that $T_3(g)/V$ is finite dimensional.

This is impossible, since $T_3(g)/V$ is a nonzero finite-dimensional module
with central character $\frac14$, while no nonzero finite dimensional
$\mathfrak{sl}(2)$-module has this central character. Therefore
$T_3(g)$ is simple.
\end{proof}

\subsection{$\mathcal U(\mathfrak h)$-dual of exponential modules and other examples}
Note that taking $U(\mathfrak h)$-dual produces further $U(\mathfrak h)$-free modules. Since all objects in $\mathcal M(1)$ are classified in Theorem \ref{classfication-rank1}, we restrict our attention to the case $\deg g = n \geq 2$, when studying the $U(\mathfrak h)$-duals of $T_1(g,b)$, $T_2(g,b)$, and $T_3(g)$.

\begin{proposition}
For $n \geq 2$ and $m \geq 1$, let  $g(t)=\sum_{i=1}^n a_i t^i$ and $q(t) = \sum_{j=1}^m b_{2j}t^{2j}$, where $a_n\neq 0$, $b_{2m}\neq 0$. The following statements hold.
\begin{itemize}
\item[(i)] If $b \in \left(\C \setminus \tfrac{1}{2}\Z\right) \cup \{\tfrac{1}{2}\}$, then $T_j(g,b)^{\vee_{\mathfrak h}}$ is simple for each $j \in \{1,2\}$.
\item[(ii)] The modules $\left(T_3(q)^{\mathrm{ev}}\right)^{\vee_{\mathfrak h}}$ and $\left(T_3(q)^{\mathrm{odd}}\right)^{\vee_{\mathfrak h}}$ are simple.
\end{itemize}
\end{proposition}

\begin{proof}
Note that $T_3(q)^{\mathrm{ev}}$, $T_3(q)^{\mathrm{odd}}$, and $T_j(g,b)$ for $j\in\{1,2\}$ are simple, since $b \in \left(\C \setminus \tfrac{1}{2}\Z\right) \cup \left\{\tfrac{1}{2}\right\}$. Moreover, by Propositions \ref{realizationT-1} and \ref{decompT-3},
$$\mathfrak f\left(T_1(g,b)\right)=\mathfrak f\left(T_2(g,b)\right)= \mathfrak f\left(T_3(q)^{\mathrm{ev}}\right)=\mathfrak f\left(T_3(q)^{\mathrm{odd}}\right)=0. $$

Hence, Proposition \ref{prop-hdual-simple-defect} implies the statements.
\end{proof}

\begin{corollary}
For $n\geq 2$, let $a_1,\dots,a_{n-1}\in\C$ and $a_n\in\C^*$. Then $M\bigl(b,(0,1,n-1),K_1\bigr)$ and $M\bigl(b,(0,n-1,1),K_2\bigr)$ are simple if and only if $b\notin 1+\tfrac12\mathbb Z_{\geq 0}$, where $K_1$ and $K_2$ are given in Proposition \ref{realizationT-1}.
\end{corollary}

\begin{proof}
By Theorem \ref{simplicityT-2} and Corollary \ref{nolowersub(1)}, $M\bigl(b,(0,1,n-1),K_1\bigr)$ and $M\bigl(b,(0,n-1,1),K_2\bigr)$ have no submodules of lower rank. The statement follows immediately from Proposition \ref{SESgeneralcase}.
\end{proof}

\end{document}